\documentclass{article}
\usepackage[top=30truemm,bottom=30truemm,left=25truemm,right=25truemm]{geometry}
\normalsize
\usepackage{amsfonts}  %
\usepackage{set space} %
\usepackage{graphicx}
\usepackage{mathptmx}  
\usepackage{latexsym}
\usepackage{latexsym}
\usepackage{amsmath}
\usepackage{amsthm} 
\usepackage[all]{xy}
\newtheorem{defi}{Definition}
\newtheorem{theo}{Theorem}
\newtheorem{lemm}{Lemma}
\newtheorem{rema}{Remark}

\title{
Proof that the real part of all non-trivial zeros of Riemann zeta function is 1/2 
}

\author{
Kimichika Fukushima
\thanks{E-mail: kimichika1a.fukushima@glb.toshiba.co.jp; km.fukushima@mx2.ttcn.ne.jp 
Phone: +81-90-4602-0490 Phone/Fax: +81-45-831-8881}\\
Advanced Reactor System Engineering Department,\\
Toshiba Nuclear Engineering Service Corporation,\\
8, Shinsugita-cho, Isogo-ku, Yokohama 235-8523, Japan
}

\date{             }

\begin{document}

\maketitle

This article proves the Riemann hypothesis, which states that all non-trivial zeros of the zeta function have a real part equal to 1/2. We inspect in detail the integral form of the (symmetrized) completed zeta function, which is a product between the zeta and gamma functions. It is known that two integral lines, expressing the completed zeta function, rotated from the real axis in the opposite directions, can be shifted without affecting the completed zeta function owing to the residue theorem. The completed zeta function is regular in the region of the complex plane under consideration. For convenience in the subsequent singularity analysis of the above integral, we first
deform and shift the integral contours.
We then investigate the singularities of the composite elements (caused by polynomial integrals in opposite directions), which appear only in the case for which the distance between the contours and the origin of the coordinates approaches zero. The real part of the zeros of the zeta function is determined to be 1/2 along a symmetry line from the singularity removal condition. (In the other points, the singularities are adequately cancelled as a whole to lead to a finite value.) 
\par
\verb+ +
\par


\section{Introduction}
\label{sec:1}

By connecting complex analysis with number theory, Riemann observed \cite{Riemann 1859} that (denoting a set of real numbers by $\mathbb{R}$ and letting $x \in \mathbb{R}$) the function $\pi(x)$, which denotes the number of prime numbers below a given number $x$, contains the summation over non-trivial zeros (points at which the function vanishes) of the zeta function.
Riemann expected (denoting a set of complex numbers by $\mathbb{C}$ and letting $z \in \mathbb{C}$) the real part of the non-trivial zeros of the zeta function ${\zeta}(z)$ to be 1/2, which is known as the Riemann hypothesis.
Furthermore, von Koch showed \cite{von Koch 1901} that $\pi(x)$ is well approximated by the offset logarithmic integral function ${\rm Li}(x)$ as
\begin{eqnarray}
\pi(x)={\rm Li}(x)+O(x^{\frac{1}{2}}\log x),
\end{eqnarray}
which is equivalent to the Riemann hypothesis.
We denote a set of natural numbers by $\mathbb{N}$ and let $n \in \mathbb{N}$ and $z \in \mathbb{C}$, then
the zeta function $\zeta(z)$ is defined as a function, which is analytically continuated in the complex plane from the expression defined below
[3-5]
\begin{eqnarray}
\label{eqn:0a}
\zeta(z)
:=
\sum_{n=1}^{\infty}\frac{1}{n^z},
\end{eqnarray}
for
$z$ that satisfies ${\rm Re}(z)>1$
(we denote the real and imaginary parts of $z$ as ${\rm Re}(z)$ and ${\rm Im}(z)$, respectively.)
The zeta function is also obtained with the help of the gamma function $\Gamma(z)$, and, letting
$t^{\prime}$
$ \in \mathbb{R}$,
then the gamma function is defined as a function that is also analytically continuated into all points in the complex plane from
[3,6-9]
\begin{eqnarray}
\label{eqn:gamd}
\Gamma(z)
:=
\int_{0}^{\infty} dt^{\prime} (t^{\prime})^{z-1} \exp(-t^{\prime}),
\end{eqnarray}
for ${\rm Re } (z)>0$.
\par
Concerning the zeros of the zeta function, which states $\zeta(z)=0$, there exist trivial zeros, such as negative integers $-2, -4, \cdot \cdot \cdot \cdot$ \cite{Moriguchi et al. 195619571960}. In contrast, Hardy showed that numerous non-trivial zeros of the zeta function exist along the line with the real part equal to 1/2 \cite{Hardy 1914}; however, not all the real parts of the non-trivial zeros are known. The work on such as imaginary parts of the zeros is reported in literature \cite{Reyna 2011}. The computational approach \cite{Gourdon 2004} strongly suggests that the real part of zeros of the zeta function is 1/2.

\par
On the other hand, 
letting $z, w \in \mathbb{C}$,
for the completed zeta function defined by
\begin{eqnarray}
\label{eqn:zeta0}
\hat{\zeta}(z)
:=
\pi^{-\frac{z}{2}}\Gamma(\frac{z}{2}) \zeta(z),
\end{eqnarray}
{\it integral form of the (completed) zeta function} is expressed as
\begin{eqnarray}
\label{eqn:zeta01}
\nonumber
\pi^{-\frac{1-z}{2}}\Gamma(\frac{1-z}{2}) \zeta(1-z)
\end{eqnarray}
\begin{eqnarray}
\nonumber
=\pi^{-\frac{1-z}{2}} \Gamma(\frac{1-z}{2}) \int_{0 \searrow 1} dw
\frac{w^{z-1}\exp(-\pi i w^2)}{\exp(\pi i w) - \exp(-\pi i w) }
\end{eqnarray}
\begin{eqnarray}
+\pi^{-\frac{z}{2}} \Gamma(\frac{z}{2}) \int_{0 \swarrow 1} dw
\frac{w^{-z}\exp(\pi i w^2)}{\exp(\pi i w) - \exp(-\pi i w) }.
\end{eqnarray}
\\
The above integral is performed along the integral lines $0 \searrow 1$ and $0 \swarrow 1$ with the slopes $-1$ and $+1$, respectively, which pass through an arbitrary point in the region between 0 and 1 of the real axis. Since the residue theorem exhibits the above equation, the integral form is independent of the shift of this intersection point between 0 and 1. Furthermore, in the original form \cite{Siegel 1932} of the above equation, the function $\Gamma(z/2)$ in the second term on the right-hand side is proportional to the regular function for ${\rm Re}(z)<1$. The function $\Gamma((1-z)/2)$ on the left-hand side is regular in the region ${\rm Re}(z)<1$, while the right-hand side is also regular because of the existence of the derivative \cite{Yoshida 1965,Mizuno 1966}, and the function $\zeta(1-z)$ is analytically continuated uniquely \cite{Yoshida 1965} into the region $0<{\rm Re}(z)<1$ (the real part of zeros of $\zeta(z)$ exists only in this region). This paper takes into account the form mentioned above.
\par
Since the gamma function is regular, the non-trivial zeros of the completed zeta function $\hat{\zeta}(z)$, which is the product between the gamma function $\Gamma(z)$ and zeta function ${\zeta}(z)$, coincide with those of the zeta function ${\zeta}(z)$ in the region being considered with the real part between 0 and 1. As is described in this paper, each of the two line integrals expressing the completed zeta function has a singularity when the integral lines approach the axis origin. However, the completed zeta function $\hat{\zeta}(z)$ does not depend on a specific point of the intersection point (shifted between 0 and 1 along the real axis) between the above integral line and the real axis due to the residue theorem, and $\hat{\zeta}(z)$ is regular in the considering region. Then, these singularities must exactly cancel each other for $\hat{\zeta}(z)=0$, which is expected to lead to the determination of the real part of the zeros of the zeta function ${\zeta}(z)$.
\par
Considering the status mentioned above, this paper is aimed at
proving the Riemann hypothesis.
We first
deform and shift the contours of
the integral (as in Figs. \ref{fig:fig1} and \ref{fig:fig2} for the integral form of the completed zeta function) along the integral line rotated from the real axis,
for
convenience in
the subsequent analysis of the singularity of the integral in a complex plane. By this 
deformation and shift of the contours for
the integral, the singularity analysis can be concentrated on the components of the integral
around the origin of coordinates.
\par
This research then addresses the singularities that appear in the two integral lines of the integral form of the completed zeta function. The singularities of the integrands for the composite elements near the origin of the real axis are caused by polynomials, only in the case when the contour-origin distance approaches zero. These
singularities
adequately cancel each other yielding a finite value independent of the integral contour as a whole. In contrast, from the equation $\hat{\zeta}(z)=0$ for completed zeta function $\hat{\zeta}(z)$, the real part of the zeros of $\hat{\zeta}(z)$ is determined by requiring these singularities to be an identical order power of the integral variable in the integrands leading to the exact singularity cancellation (given by Theorem \ref{thm:ReH}). This requirement results in a value of 1/2 for the real part of zeros of the completed zeta function
$\hat{\zeta}(z)$ 
(and the original zeta function ${\zeta}(z)$) due to the symmetry with respect to the 1/2 real part, which is the originality of the present study and
proves
the Riemann hypothesis.
The Riemann hypothesis is one of the most important unproved problems in mathematics, and has its equivalent and advanced
(extended)
conjectures in other related fields. The positive
proof
of the Riemann hypothesis advances mathematics in other related fields \cite{Sato 2007,Kurokawa 2009}.
\par
The contents of this paper are as follows. Section \ref{sec:2} describes the
deformation and shift of the contours for
the integral in the integral form of the completed zeta function
for convenience in
the subsequent singularity analysis of the integral. Section \ref{sec:3} presents the 
proof that the real part
of 
all non-trivial
zeros of the zeta function
is equal to 1/2, as was conjectured by Riemann,
followed by the conclusion.
\par

\section{
Deformation and shift of the contours for the integral form of the completed zeta function
for the singularity analysis in a complex plane
} 
\label{sec:2}

This section presents the
deformation and shift of the contours for
the integral (in the integral form of the completed zeta function) along the line rotated from the real axis
for convenience in
the subsequent analysis (in Section 3) of the singularity of the integral in a complex plane.
In this section, we first convert the integral form of the completed zeta function expressed by Eq. (\ref{eqn:zeta01}) to the usual form and thus obtain Theorem \ref{thm:czta}. 
Moreover, we define the radii centered at the (coordinate) origin, the main points and the contours (in Definitions \ref{dfi:intr}, \ref{dfi:intw} with Figs. \ref{fig:fig1}, \ref{fig:fig2}) in the complex plane.
Then, the contours denoted by $0 \searrow 1$ and $0 \swarrow 1$ in the integral form of the (completed) zeta function in Eq. (\ref{eqn:zeta01}) are deformed around the origin of coordinates to the arcs in Figs. \ref{fig:fig1} and \ref{fig:fig2}, respectively (in Lemma \ref{lem:Defm}). Subsequently, the remaining straight-line parts of the contours are shifted to point to the (coordinate) origin. Finally, we separate the finite integrals (in Lemmas \ref{lem:IntLM}, \ref{lem:IntMm}) along the shifted contours, which have sufficient distance to the (coordinate) origin, from the integrals around the origin containing the singularities, which appear only when the contours approach the origin.

Notations used in this paper are as follows. Let $z, v \in \mathbb{C}$ and $x, y \in \mathbb{R}$, and let $i$ be the imaginary unit, then
\begin{eqnarray}
(x,y)
:=
z=x+iy.
\end{eqnarray}
We denote the real and imaginary components as
\begin{eqnarray}
\label{eqn:re}
z_{\rm R}
:=
{\rm Re } (z)=x,
\hspace{2ex}
z_{\rm I}
:=
{\rm Im} (z)=y,
\hspace{2ex}
v_{\rm R}:={\rm Re } (v),
\hspace{2ex}
v_{\rm I}:={\rm Im} (v)
\hspace{2ex}
\mbox{ with }
\hspace{2ex}
v=-z, z-1.
\end{eqnarray}
\par
The usual integral form of the completed zeta function is as follows.

\begin{theo}
\label{thm:czta}
(the (third) integral form of the (completed) zeta function) Let
$z, w \in \mathbb{C}$.
Let $\hat{\zeta}(z)$ be the completed zeta function defined by Eq. (\ref{eqn:zeta0}).  Let
\begin{eqnarray}
\label{eqn:zetl}
\hat{\zeta}_{\rm l}(z)
:=
\pi^{-\frac{z}{2}} \Gamma(\frac{z}{2}) \int_{0 \searrow 1} dw
\frac{w^{-z}\exp(-\pi i w^2)}{\exp(\pi i w) - \exp(-\pi i w) },
\end{eqnarray}
and let 
\begin{eqnarray}
\label{eqn:zetr}
\hat{\zeta}_{\rm r}(z)
:=
\pi^{-\frac{1-z}{2}} \Gamma(\frac{1-z}{2}) \int_{0 \swarrow 1} dw
\frac{w^{z-1}\exp(\pi i w^2)}{\exp(\pi i w) - \exp(-\pi i w) },
\end{eqnarray}
in terms of the gamma function $\Gamma(z)$. Then $\hat{\zeta}(z)$ is expressed by
\begin{eqnarray}
\label{eqn:zet0}
\hat{\zeta}(z)=\hat{\zeta}_{\rm l}(z)
+\hat{\zeta}_{\rm r}(z),
\end{eqnarray}
which is called the (third) integral form of the (completed) zeta function.
\end{theo}

\begin{proof}
From
Eq. (\ref{eqn:zeta01}), 
the above completed zeta function $\hat{\zeta}(z)$ is obtained by the exchange $z \leftrightarrow 1-z$, where the region $0<{\rm Re}(z)<1$ is kept under this exchange. The above region $0<{\rm Re}(z)<1$ is consistent with the region $0<{\rm Re}(z)<1$ we are considering in this paper. 
\end{proof}

In addition, the completed zeta function satisfies the following known symmetry relation
\cite{Riemann 1859,Kurokawa 2009,Sato 2007}
 for the exchange $z \leftrightarrow 1-z$
\begin{eqnarray}
\label{eqn:zets}
\hat{\zeta}(z)=\hat{\zeta}(1-z).
\end{eqnarray}
\par

\begin{figure}[tbh]
\begin{center}
\setlength{\unitlength}{1 mm}
\begin{picture}(80,115)(-40,-80)
\put( -40,   0 ){\vector(1,0){80}}
\put(   0,  -37){\vector(0,1){77}}
\put(-7, 7){\line(-1, 1){28}}
\put( 7,-7){\line( 1,-1){28}}
\qbezier(-7,7)( 0,14)(7, 7)
\qbezier( 7,7)(14, 0)(7,-7)
\put( -3., -3.){${\rm O}$}
\put( 35.,  1.){${\rm Re}$}
\put( -5., 35.){${\rm Im}$}
\put(-35, 30){$\searrow$}
\put( -6,  6){$\nearrow$}
\put(2.0,2.0){$\searrow$}
\put(  5, -5){$\swarrow$}
\put(  9,-14){$\searrow$}
\put(-31.5,  25){$w_{\rm l1}$}
\put(-21.5,  15){$w_{\rm l2}$}
\put(-11.5, 4.5){$w_{\rm l3}$}
\put( 3.5,  4.0){$w_{\rm l4}$}
\put(  2.5,-9.5){$w_{\rm l5}$}
\put( 13.5, -20){$w_{\rm l6}$}
\put( 23.5, -30){$w_{\rm l7}$}
\put(-22, 35){${\rm \tilde{C}_{l}}$}
\put(-31, 32){${\rm \tilde{C}_{l1}}$}
\put(-21, 22){${\rm \tilde{C}_{l2}}$}
\put(-11, 12){${\rm \tilde{C}_{lp}}$}
\put(  7,  8){${\rm \tilde{C}_{lc}}$}
\put( 13,-12){${\rm \tilde{C}_{ln}}$}
\put( 23,-22){${\rm \tilde{C}_{l3}}$}
\put(31.5,-31){${\rm \tilde{C}_{l4}}$}
\put(-39, -42){$w_{\rm l1}$:
   $(-\frac{\sqrt{2}}{2}r_{\rm M},\frac{\sqrt{2}}{2}r_{\rm M})$}
\put( 1, -42){$w_{\rm l2}$:
   $(-\frac{\sqrt{2}}{2}r_{\rm m},\frac{\sqrt{2}}{2}r_{\rm m})$}
\put(-39, -48){$w_{\rm l3}$:
   $(-\frac{\sqrt{2}}{2}r_{\rm 1},\frac{\sqrt{2}}{2}r_{\rm 1})$}
\put( 1, -48){$w_{\rm l4}$:
   $( \frac{\sqrt{2}}{2}r_{\rm 1},\frac{\sqrt{2}}{2}r_{\rm 1})$}
\put(-39, -54){$w_{\rm l5}$:
   $(\frac{\sqrt{2}}{2}r_{\rm 1},-\frac{\sqrt{2}}{2}r_{\rm 1})$}
\put(-39, -60){$w_{\rm l6}$:
   $(\frac{\sqrt{2}}{2}r_{\rm m},-\frac{\sqrt{2}}{2}r_{\rm m})$}
\put( 1, -60){$w_{\rm l7}$:
   $(\frac{\sqrt{2}}{2}r_{\rm M},-\frac{\sqrt{2}}{2}r_{\rm M})$}
\put(-39,-65){${\rm \tilde{C}_l}$:
   contour composed of contours 
   from ${\rm \tilde{C}_{l1}}$ to ${\rm \tilde{C}_{l4}}$}
\put(-39,-70){${\rm \tilde{C}_{l1}}$-${\rm \tilde{C}_{l4}}$:
   contours defined below}
\end{picture}
\caption{Location of complex numbers (denoted at the bottom of the figure) and contours in the complex $w$-plane. 
Contour ${\rm \tilde{C}_{l}}$ comprises ${\rm \tilde{C}_{l1}}$ (from $[\exp(\frac{3}{4}\pi i)]\infty$ to $w_{\rm l1}$), ${\rm \tilde{C}_{l2}}$ (from $w_{\rm l1}$ to $w_{\rm l2}$), ${\rm \tilde{C}_{lp}}$ (from $w_{\rm l2}$ to $w_{\rm l3}$), ${\rm \tilde{C}_{lc}}$ (from $w_{\rm l3}$ to $w_{\rm l5}$ via $w_{\rm l4}$), ${\rm \tilde{C}_{ln}}$ (from $w_{\rm l5}$ to $w_{\rm l6}$), ${\rm \tilde{C}_{l3}}$ (from $w_{\rm l6}$ to $w_{\rm l7}$) and ${\rm \tilde{C}_{l4}}$ (from $w_{\rm l7}$ to $[\exp(\frac{-1}{4}\pi i)]\infty$).
} 
\label{fig:fig1}
\end{center}
\end{figure}
\begin{figure}[tbh]
\begin{center}
\setlength{\unitlength}{1 mm}
\begin{picture}(80,115)(-40,-80)
\put( -40,   0 ){\vector(1,0){80}}
\put(   0,  -36){\vector(0,1){76}}
\put( 7, 7){\line( 1, 1){28}}
\put(-7,-7){\line(-1,-1){28}}
\qbezier( 7, 7)(14,  0)( 7,-7)
\qbezier( 7,-7)( 0,-14)(-7,-7)
\put( -3., -3.){${\rm O}$}
\put( 35.,  1.){${\rm Re}$}
\put( -5., 35.){${\rm Im}$}
\put( 31, 30){$\swarrow$}
\put(5.5, 2.0){$\searrow$}
\put(4.5,-6.5){$\swarrow$}
\put(-4.5,-8.0){$\nwarrow$}
\put( -13,-14){$\swarrow$}
\put( 27.5,  25)  {$w_{\rm r1}$}
\put( 17.5,  15)  {$w_{\rm r2}$}
\put(  8.5, 5.75) {$w_{\rm r3}$}
\put(   7.5,-9.5) {$w_{\rm r4}$}
\put( -8.5,-11.25){$w_{\rm r5}$}
\put(-17.5, -20)  {$w_{\rm r6}$}
\put(-27.5, -30)  {$w_{\rm r7}$}
\put(  17, 35){${\rm \tilde{C}_{r}}$}
\put( 26, 32){${\rm \tilde{C}_{r1}}$}
\put( 16, 22){${\rm \tilde{C}_{r2}}$}
\put( 6, 12){${\rm \tilde{C}_{rp}}$}
\put(1.5,-4.0){${\rm \tilde{C}_{rc}}$}
\put(-18,-12){${\rm \tilde{C}_{rn}}$}
\put(-28,-22){${\rm \tilde{C}_{r3}}$}
\put(-36.5,-31){${\rm \tilde{C}_{r4}}$}
\put(-39, -41){$w_{\rm r1}$:
   $(\frac{\sqrt{2}}{2}r_{\rm M},\frac{\sqrt{2}}{2}r_{\rm M})$}
\put( 1, -41){$w_{\rm r2}$:
   $(\frac{\sqrt{2}}{2}r_{\rm m},\frac{\sqrt{2}}{2}r_{\rm m})$}
\put(-39, -47){$w_{\rm r3}$:
   $(\frac{\sqrt{2}}{2}r_{\rm 1},\frac{\sqrt{2}}{2}r_{\rm 1})$}
\put( 1, -47){$w_{\rm r4}$:
   $( \frac{\sqrt{2}}{2}r_{\rm 1},-\frac{\sqrt{2}}{2}r_{\rm 1})$}
\put(-39, -53){$w_{\rm r5}$:
   $(-\frac{\sqrt{2}}{2}r_{\rm 1},-\frac{\sqrt{2}}{2}r_{\rm 1})$}
\put(-39, -59){$w_{\rm r6}$:
   $(-\frac{\sqrt{2}}{2}r_{\rm m},-\frac{\sqrt{2}}{2}r_{\rm m})$}
\put( 1, -59){$w_{\rm r7}$:
   $(-\frac{\sqrt{2}}{2}r_{\rm M},-\frac{\sqrt{2}}{2}r_{\rm M})$}
\put(-39,-65){${\rm \tilde{C}_r}$:
   contour composed of contours 
   from ${\rm \tilde{C}_{r1}}$ to ${\rm \tilde{C}_{r4}}$}
\put(-39,-70){${\rm \tilde{C}_{r1}}$-${\rm \tilde{C}_{r4}}$:
   contours defined below}
\end{picture}
\caption{Location of complex numbers (denoted at the bottom of the figure) and contours in the complex $w$-plane.
Contour ${\rm \tilde{C}_{r}}$ comprises ${\rm \tilde{C}_{r1}}$ (from $[\exp({\frac{1}{4}\pi i})]\infty$ to $w_{\rm r1}$), ${\rm \tilde{C}_{r2}}$ (from $w_{\rm r1}$ to $w_{\rm r2}$), ${\rm \tilde{C}_{rp}}$ (from $w_{\rm r2}$ to $w_{\rm r3}$), ${\rm \tilde{C}_{rc}}$ (from $w_{\rm r3}$ to $w_{\rm r5}$ via $w_{\rm r4}$), ${\rm \tilde{C}_{rn}}$ (from $w_{\rm r5}$ to $w_{\rm r6}$), ${\rm \tilde{C}_{r3}}$ (from $w_{\rm r6}$ to $w_{\rm r7}$) and ${\rm \tilde{C}_{r4}}$ (from $w_{\rm r7}$ to $[\exp({\frac{-3}{4}\pi i})]\infty$).
} 
\label{fig:fig2}
\end{center}
\end{figure}

To evaluate the integrals 
of the completed zeta function
(in 
Eqs. (\ref{eqn:zetl})-(\ref{eqn:zet0})),
we
further define the detailed integrands,
the main points
(in the complex plane) and
the deformed and shifted contours of the integrals
for use in the subsequent
processes.
\par

\begin{defi}
\label{dfi:intr}
Let $w \in \mathbb{C}$, and let $r_{1{\rm l}}, r_{1{\rm r}}, r_{1}, r_{\rm m} \in \mathbb{R}$ with $0<r_{1{\rm l}}, r_{1{\rm r}}, r_{1} << r_{\rm m}<1/2$.
Then, the specific radii $r_{1{\rm l}}, r_{1{\rm r}}, r_{1}$ and $r_{\rm m}$ of $w$, centered at the origin of the complex $w$-plane, are defined to be small enough so that the follwoing denominator, denoted as $I^{({\rm De})}$, and parts of the numerators, denoted as $I^{({\rm Nu})-}$ and $I^{({\rm Nu})+}$, in the integrands in Eqs. (\ref{eqn:zetl})-(\ref{eqn:zet0}), are approximated by
\begin{eqnarray}
\nonumber
\{
I^{({\rm De})}=\exp(\pi iw)-\exp(-\pi iw) \approx 2\pi i w
\},
\hspace{2ex}
\{
I^{({\rm Nu})-}=\exp(-\pi iw^2) \approx 1
\}
\hspace{2ex}\mbox{ and }
\end{eqnarray}
\begin{eqnarray}
\{
I^{({\rm Nu})+}=\exp(+\pi iw^2) \approx 1
\}
\hspace{2ex}
\mbox{ for } |w| \leq r_{1{\rm l}}, r_{1{\rm r}} r_{1}
\mbox{ and } |w| \leq r_{\rm m}.
\end{eqnarray}
Furthermore, let $\theta \in \mathbb{R}$, with $\theta=\frac{\pm 1}{4}\pi, \frac{ \pm 3}{4}\pi$, be the angle (argument) of $w$ measured counterclockwise from the real axis in the complex plane. Let $r_{\rm M} \in \mathbb{R}$ be the specific radius of $w$, centered at the (coordinate) origin, and defiend to be large enough so that the following denominator, denoted as $I^{(De)}$, in the integrands in Eqs. (\ref{eqn:zetl})-(\ref{eqn:zet0}) is approximated by
\begin{eqnarray}
\label{eqn:ConRM}
I^{(\rm {\rm De})}=
\exp(\pi iw)-\exp(-\pi iw) \approx 
\exp(\pi iw) \hspace{2ex}\mbox{ or }\hspace{2ex} -\exp(-\pi iw).
\end{eqnarray}
(The condition on the radius $r_{\rm M}$ is described in detail later around Eqs.
(\ref{eqn:DeLM})-(\ref{eqn:DeA3}) ).
\end{defi}

Here, we define the complex numbers in the complex $w$-plane shown in Figs. \ref{fig:fig1} and \ref{fig:fig2}.

\begin{defi}
\label{dfi:intw}
Let $w \in \mathbb{C}$, and let $w_{\rm l1}, w_{\rm l2}, w_{\rm l3}, w_{\rm l4}, w_{\rm l5}, w_{\rm l6}, w_{\rm l7} \in \mathbb{C}$. Using the specific radii $r_1$ and $r_{\rm 1l}$ (of $w$) and setting $r_1=r_{\rm 1l}$ in Definition \ref{dfi:intr},
we define the above complex numbers, whose locations in the complex $w$-plane are shown in Fig. \ref{fig:fig1}, by
\begin{eqnarray}
\label{eqn:wl1}
\nonumber
w_{\rm l1}:=r_{\rm M} \exp(\frac{3}{4}\pi i)
=(r_{\rm M}\cos(\frac{3}{4}\pi),r_{\rm M}\sin(\frac{3}{4}\pi)),
\hspace{4ex}
w_{\rm l2}:=r_{\rm m} \exp(\frac{3}{4}\pi i)
=(r_{\rm m}\cos(\frac{3}{4}\pi),r_{\rm m}\sin(\frac{3}{4}\pi)),
\end{eqnarray}
\begin{eqnarray}
\nonumber
w_{\rm l3}:=r_{\rm 1l} \exp(\frac{3}{4}\pi i)
=r_{\rm 1} \exp(\frac{3}{4}\pi i)
=(r_{\rm 1}\cos(\frac{3}{4}\pi),r_{\rm 1}\sin(\frac{3}{4}\pi)),
\hspace{4ex}
w_{{\rm l}4}:=(r_{\rm 1}\cos(\frac{1}{4}\pi),r_{\rm 1}\sin(\frac{1}{4}\pi)),
\end{eqnarray}
\begin{eqnarray}
\nonumber
w_{\rm l5}:=r_{\rm 1} \exp(\frac{-1}{4}\pi i)
=(r_{\rm 1}\cos(\frac{-1}{4}\pi),r_{\rm 1}\sin(\frac{-1}{4}\pi)),
\end{eqnarray}
\begin{eqnarray}
\nonumber
w_{\rm l6}:=r_{\rm m}\exp(\frac{-1}{4}\pi i)
=(r_{\rm m}\cos(\frac{-1}{4}\pi),r_{\rm m}\sin(\frac{-1}{4}\pi)),
\hspace{4ex}
w_{\rm l7}:=r_{\rm M}\exp(\frac{-1}{4}\pi i)
=(r_{\rm M}\cos(\frac{-1}{4}\pi),r_{\rm M}\sin(\frac{-1}{4}\pi)).
\end{eqnarray}
\begin{eqnarray}
\end{eqnarray}
\par 
Similarly, 
let $w_{\rm r1}, w_{\rm r2}, w_{\rm r3}, w_{\rm r4}, w_{\rm r5}, w_{\rm r6}, w_{\rm r7} \in \mathbb{C}$.
Using the specific radii (of $w$) $r_1, r_{\rm 1r}$ and setting $r_1=r_{\rm 1r}$ in Definition \ref{dfi:intr},
we define the above complex numbers, whose locations in the complex $w$-plane are shown in Fig. \ref{fig:fig2}, by
\begin{eqnarray}
\label{eqn:wr1}
\nonumber
w_{\rm r1}:=r_{\rm M} \exp(\frac{1}{4}\pi i)
=(r_{\rm M}\cos(\frac{1}{4}\pi),r_{\rm M}\sin(\frac{1}{4}\pi)),
\hspace{4ex}
w_{\rm r2}:=r_{\rm m} \exp(\frac{1}{4}\pi i)
=(r_{\rm m}\cos(\frac{1}{4}\pi),r_{\rm m}\sin(\frac{1}{4}\pi)),
\end{eqnarray}
\begin{eqnarray}
\nonumber
w_{\rm r3}:=r_{\rm 1r} \exp(\frac{1}{4}\pi i)
=r_{\rm 1} \exp(\frac{1}{4}\pi i)
=(r_1\cos(\frac{1}{4}\pi),r_1\sin(\frac{1}{4}\pi)),
\hspace{4ex}
w_{\rm r4}:=(r_{\rm 1}\cos(\frac{-1}{4}\pi),r_{\rm 1}\sin(\frac{-1}{4}\pi)),
\end{eqnarray}
\begin{eqnarray}
\nonumber
w_{\rm r5}:=r_{\rm 1} \exp(\frac{-3}{4}\pi i)
=(r_{\rm 1} \cos(\frac{-3}{4}\pi),r_{\rm 1} \sin(\frac{-3}{4}\pi)),
\end{eqnarray}
\begin{eqnarray}
\nonumber
w_{\rm r6}:=r_{\rm m} \exp(\frac{-3}{4}\pi i)
=(r_{\rm m}\cos(\frac{-3}{4}\pi),r_{\rm m}\sin(\frac{-3}{4}\pi)),
\hspace{4ex}
w_{\rm l7}:=r_{\rm M}\exp(\frac{-3}{4}\pi i)
=(r_{\rm M}\cos(\frac{-3}{4}\pi),r_{\rm M}\sin(\frac{-3}{4}\pi)).
\end{eqnarray}
\begin{eqnarray}
\end{eqnarray}
\end{defi}
\par
We now define the deformed and shifted contours of the integrals in the completed zeta function.

\begin{defi}
\label{dfi:Conlr}
Using the complex numbers $w_{\rm l1}$-$w_{\rm l7}$ in Definition \ref{dfi:intw} (points in the complex $w$-plane), the contours in Fig. \ref{fig:fig1} are defined as follows:
\begin{itemize}
\item
${\rm \tilde{C}_{l}}$: contour composed of the contours from ${\rm \tilde{C}_{l1}}$ to ${\rm \tilde{C}_{l4}}$ (${\rm \tilde{C}_{l1}}, {\rm \tilde{C}_{l2}}, {\rm \tilde{C}_{lp}}, {\rm \tilde{C}_{lc}}, {\rm \tilde{C}_{ln}}, {\rm \tilde{C}_{l3}}$, ${\rm \tilde{C}_{l4}}$),
\item
${\rm \tilde{C}_{l1}}$: straight-line contour from $[\exp(\frac{3}{4} \pi i)]\infty$ to $w_{\rm l1}$ (with radius $r_{\rm M}$) in the direction of the arrow,
\item
${\rm \tilde{C}_{l2}}$: straight-line contour from
$w_{\rm l1}$ (with radius $r_{\rm M}$) to 
$w_{\rm l2}$ (with radius $r_{\rm m}$),
\item
${\rm \tilde{C}_{lp}}$: straight-line contour from 
$w_{\rm l2}$ (with radius $r_{\rm m}$) to
$w_{\rm l3}$ (with radius $r_{1}$),
\item
${\rm \tilde{C}_{lc}}$: arc (of circle) contour from 
$w_{\rm l3}$ to $w_{\rm l5}$ via $w_{\rm l4}$ (in the direction of the arrow) centered at the (coordinate) origin $(0,0)$ with the radius $r_{1}$,
\item
${\rm \tilde{C}_{ln}}$: straight-line contour from 
$w_{\rm l5}$ (with radius $r_{\rm 1}$)
to $w_{\rm l6}$  (with radius $r_{\rm m}$) in the direction of the arrow,
\item
${\rm \tilde{C}_{l3}}$: straight-line contour from 
$w_{\rm l6}$ (with radius $r_{\rm m}$) to 
$w_{\rm l7}$ (with radius $r_{\rm M}$),
\item
${\rm \tilde{C}_{l4}}$ straight-line contour from 
$w_{\rm l7}$  (with radius $r_{\rm M}$) to
$[\exp(\frac{-1}{4} \pi i)]\infty$.
\end{itemize}
\par
Similarly, using the complex numbers $w_{\rm r1}$-$w_{\rm r7}$ in Definition \ref{dfi:intw} (points in the complex $w$-plane), contours in Fig. \ref{fig:fig2} are defined as follows:
\begin{itemize}
\item
${\rm \tilde{C}_{r}}$: contour composed of the contours from ${\rm \tilde{C}_{r1}}$ to ${\rm \tilde{C}_{r4}}$ (${\rm \tilde{C}_{r1}}, {\rm \tilde{C}_{r2}}, {\rm \tilde{C}_{rp}}, {\rm \tilde{C}_{rc}}, {\rm \tilde{C}_{rn}}, {\rm \tilde{C}_{r3}}$, ${\rm \tilde{C}_{r4}}$),
\item
${\rm \tilde{C}_{r1}}$: straight-line contour from $[\exp(\frac{1}{4} \pi i)]\infty$ to
$w_{{\rm r}1}$ (with radius $r_{\rm M}$) in the direction of the arrow,
\item
${\rm \tilde{C}_{r2}}$: straight-line contour from 
$w_{\rm r1}$ (with radius $r_{\rm M}$) to 
$w_{\rm r2}$ (with radius $r_{\rm m}$),
\item
${\rm \tilde{C}_{rp}}$: straight-line contour from 
$w_{\rm r2}$ (with radius $r_{\rm m}$) to 
$w_{\rm r3}$ (with radius $r_{1}$),
\item
${\rm \tilde{C}_{rc}}$: arc (of circle) contour from $w_{\rm r3}$ to $w_{\rm r5}$ via $w_{\rm r4}$ (in each direction of the arrow) centered at the (coordinate) origin $(0,0)$ with the radius $r_{1}$,
\item
${\rm \tilde{C}_{rn}}$: straight-line contour from 
$w_{\rm r5}$ (with radius $r_{1}$) to 
$w_{\rm r6}$ (with radius $r_{\rm m}$) in the direction of the arrow,
\item
${\rm \tilde{C}_{r3}}$: straight-line contour from 
$w_{\rm r6}$ (with radius $r_{\rm m}$) to 
$w_{\rm r7}$ (with radius $r_{\rm M}$),
\item
${\rm \tilde{C}_{r4}}$ straight-line contour from 
$w_{\rm r7}$ (with radius $r_{\rm m}$) to 
$[\exp(\frac{-1}{4} \pi i)]\infty$.
\end{itemize}
\end{defi}
\par
Here, we show that it is possible to deform and shift the contours in 
Eqs. (\ref{eqn:zetl})-(\ref{eqn:zet0}) 
to the contours in Figs. \ref{fig:fig1} and \ref{fig:fig2}. 
\par
\begin{lemm}
\label{lem:Defm}
Let $a_{\rm 0 l}, a_{\rm 0 r}, a_0 \in \mathbb{R}$ be positive finite numbers between 0 and 1. Let $0 \searrow 1$ be the contour (with the slope -1), which was used in Eqs. (\ref{eqn:zeta01}), (\ref{eqn:zetl})-(\ref{eqn:zet0}) and intersects the real axis (in the complex plane) at $(a_0,0)$, with $a_0=a_{\rm 0l}$, whereas let $0 \swarrow1$ be the contour (with the slope +1) which intersects the real axis at $(a_0,0)$ with $a_0=a_{\rm 0r}$. The contour $0 \searrow 1$ can be deformed and shifted to the contour $\tilde{C_{\rm l}}$ in Definition \ref{dfi:Conlr} with Fig. \ref{fig:fig1}, whereas the contour $0 \swarrow 1$ can be deformed and shifted to the contour $\tilde{C_{\rm r}}$ in Definition \ref{dfi:Conlr} with Fig. \ref{fig:fig2}.
\end{lemm}
\par
\begin{proof}
Since the integral form of the completed zeta function in
Eqs. (\ref{eqn:zeta01}) and (\ref{eqn:zetl})-(\ref{eqn:zet0})
is derived from the residue theorem, the contour $0 \searrow 1$ can be deformed and shifted to the contour $\tilde{C_{\rm l}}$, while the contour $0 \swarrow 1$ can be deformed and shifted to the contour $\tilde{C_{\rm r}}$.
\end{proof}
\par

Using Definitions \ref{dfi:intr}-\ref{dfi:Conlr} and Lemma \ref{lem:Defm}, we prove the following lemma, which shows that the integrals in Eqs. (\ref{eqn:zetl})-(\ref{eqn:zet0}) along the contours for the regions with large distance to the (coordinate) origin are finite.

\begin{lemm}
\label{lem:IntLM}
Let $w,v, z\in \mathbb{C}$ (with $v=-z, z-1$), and let $z_{\rm R}={\rm Re}(z)$ with $0<z_{\rm R}<1$.
Let $r_{\rm M} \in \mathbb{R}$ be the large (lower bound of) radius (in Definitions \ref{dfi:intr}, \ref{dfi:intw}) of $w$ along the shifted straight-line contours.
Let ${\rm \tilde{C}_{lh}}$ be the contour, which is either of the contours denoted by ${\rm \tilde{C}_{l1}} and {\rm \tilde{C}_{l4}}$ (in Fig. \ref{fig:fig1}), while let ${\rm \tilde{C}_{rh}}$ be either of the contours ${\rm \tilde{C}_{r1}} and {\rm \tilde{C}_{r4}}$ (in Fig. \ref{fig:fig2}).
Then, the following integrals of the integrands in Eqs. (\ref{eqn:zetl})-(\ref{eqn:zet0})
\begin{eqnarray}
\label{eqn:Shrl}
\nonumber
I^{\rm S}_{\tilde{\rm C}_{\rm lh}}
    =\int_{\tilde{\rm C}_{\rm lh}} dw
\frac{w^{-z}\exp(-\pi i w^2)}{\exp(\pi iw)-\exp(-\pi iw)},
\hspace{4ex}
I^{\rm S}_{\tilde{\rm C}_{\rm rh}}
    =\int_{\tilde{\rm C}_{\rm rh}} dw
\frac{w^{z-1}\exp(+\pi i w^2)}{\exp(\pi iw)-\exp(-\pi iw)}
\end{eqnarray}
\begin{eqnarray}
\hspace{20ex}
\mbox{ along the
contours }
\tilde{\rm C}_{\rm lh}=\tilde{\rm C}_{\rm l1}, \tilde{\rm C}_{\rm l4}
\mbox{ (in Fig. \ref{fig:fig1}) and }
\tilde{\rm C}_{\rm rh}=\tilde{\rm C}_{\rm r1}, \tilde{\rm C}_{\rm r4}
\mbox{ (in Fig. \ref{fig:fig2})},
\end{eqnarray}
are finite (negligible compared with those with singularities around the origin of coordinates).
\end{lemm}

\begin{proof}
Letting $v=-z, z-1$, the polynomial $I^{(\rm Po)}_{v}$ in the numerators of the integrands in (above) Eq. (\ref{eqn:Shrl}) denoted by
\begin{eqnarray}
I^{(\rm Po)}_v=w^{v} \hspace{2ex}\mbox{ (with } v=-z, z-1\mbox{)},
\end{eqnarray}
is rewritten (with $v_{\rm R}={\rm Re}(v), v_{\rm R}={\rm Im}(v)$) as
\begin{eqnarray}
\label{eqn:Poh}
\nonumber
I^{(\rm Po)}_{v}=w^{v_{\rm R}}w^{iv_{\rm I}}
=w^{v_{\rm R}}\exp\{\ln[w^{iv_{\rm I}}]\}
=w^{v_{\rm R}}\exp[iv_{\rm I}\ln(w)]
\end{eqnarray}
\begin{eqnarray}
\label{eqn:nwzl}
=w^{v_{\rm R}} \exp\{iv_{\rm I}[\ln(|w|)+i\arg(w)]\}
=w^{v_{\rm R}} \exp\{iv_{\rm I}\ln(|w|)-v_{\rm I}\arg(w)]\},
\end{eqnarray}
where $\arg(w)$ is argument (angle of $w$ measured counterclockwise from the real axis in the complex $w$-plane), which is restricted to the principal value between $-\pi$ and $+\pi$.
Letting  $\theta \in \mathbb{R}$ be the angle of $w$ (that is, $\theta=\arg(w)$) along the straight-line contour, then
\begin{eqnarray}
\label{eqn:Angt1}
\theta= \frac{3}{4}\pi \hspace{2ex}\mbox{for contour ${\rm \tilde{C}_{l1}}$},
\hspace{4ex}
\theta=\frac{-1}{4}\pi \hspace{2ex}\mbox{for contour ${\rm \tilde{C}_{l4}}$},
\end{eqnarray}
\begin{eqnarray}
\label{eqn:Angt2}
\theta= \frac{1}{4}\pi \hspace{2ex}\mbox{for contour ${\rm \tilde{C}_{r1}}$},
\hspace{4ex}
\theta=\frac{-3}{4}\pi \hspace{2ex}\mbox{for contour ${\rm \tilde{C}_{r4}}$}.
\end{eqnarray}
Using the above angle, the variable $w$ is expressed by
\begin{eqnarray}
\label{eqn:wthe}
w=|w|\exp(i\theta)
\hspace{2ex}\mbox{ with }\theta=\arg(w),
\end{eqnarray}
where $|w|$ is the radius (modulus) and $\theta$ is the angle (argument).
Then, from Eqs. (\ref{eqn:Poh})-(\ref{eqn:wthe}), we have
\begin{eqnarray}
\label{eqn:Ivn}
\nonumber
I^{(\rm Po)}_v
=|w|^{v_{\rm R}}\exp(iv_{\rm R}\theta)
\exp[iv_{\rm I}\ln(|w|)]\exp[-v_{\rm I}\arg(w)]
\end{eqnarray}
\begin{eqnarray}
=|w|^{v_{\rm R}}\exp(iv_{\rm R}\theta)
\exp[iv_{\rm I}\ln(|w|)]\exp(-v_{\rm I}\theta).
\end{eqnarray}
The absolute value of $I^{(\rm Po)}_v$ in (above) Eq. (\ref{eqn:Ivn}) is
\begin{eqnarray}
\label{eqn:PwAb}
|I^{(\rm Po)}_v|=|w|^{v_{\rm R}} \exp(-v_{\rm I}\theta).
\end{eqnarray}
\par
Moreover, let $I^{({\rm Nu})-}$ and $I^{({\rm Nu})+}$ be the parts of the numerators in the integrands in Eq. (\ref{eqn:Shrl}) written by
\begin{eqnarray}
I^{({\rm Nu})-}=\exp(-\pi i w^2),
\hspace{4ex}
I^{({\rm Nu})+}=\exp(+\pi i w^2).
\end{eqnarray}
Using Eqs. (\ref{eqn:Angt1})-(\ref{eqn:wthe}), it follows that
\begin{eqnarray}
I^{{(\rm Nu})-}
=\exp [-\pi i|w|^2 (\cos 2\theta +i\sin 2\theta) ]
\hspace{2ex}
\mbox{ with } \theta=\frac{3}{4}\pi,\frac{-1}{4}\pi,
\end{eqnarray}
\begin{eqnarray}
I^{{(\rm Nu})+}
=\exp [\pi i|w|^2 (\cos 2\theta +i\sin 2\theta) ]
\hspace{2ex}
\mbox{ with } \theta=\frac{1}{4}\pi,\frac{-3\pi}{4}\pi.
\end{eqnarray}
We then have
\begin{eqnarray}
\label{eqn:mMMin}
|I^{({\rm Nu})-}| =\exp (\pi |w|^2\sin 2\theta)
\hspace{2ex}
\mbox{ for }
\theta=\frac{3}{4}\pi,\frac{-1}{4}\pi,
\end{eqnarray}
\begin{eqnarray}
\label{eqn:mMPla}
|I^{({\rm Nu})+}| =\exp (-\pi |w|^2\sin 2\theta)
\hspace{2ex}
\mbox{ for }
\theta=\frac{1}{4}\pi,\frac{-3}{4}\pi.
\end{eqnarray}
Therefore, (above) Eqs. (\ref{eqn:mMMin})-(\ref{eqn:mMPla}) are reduced to
\begin{eqnarray}
\label{eqn:NuAB}
|I^{{\rm Nu})\mp}|
=\exp (-\pi |w|^2|\sin 2\theta| )
\hspace{2ex}
\mbox{ with }
\theta=\frac{\pm 1}{4}\pi,\frac{\pm 3}{4}\pi.
\end{eqnarray}
\par
In contrast, by using Eqs. (\ref{eqn:Angt1})-(\ref{eqn:wthe}) for the following denominator $I^{\rm (De)}$ in  Eq. (\ref{eqn:Shrl})
\begin{eqnarray}
\label{eqn:DeLM}
I^{\rm (De)}=\exp(\pi iw)-\exp(-\pi iw),
\end{eqnarray}
we get
\begin{eqnarray}
\label{eqn:DeA1}
\nonumber
I^{\rm (De)}=
 \exp [ \pi i|w|(\cos \theta +i\sin \theta )]
-\exp [-\pi i|w|(\cos \theta +i\sin \theta )]
\end{eqnarray}
\begin{eqnarray}
=\exp ( \pi i|w|\cos \theta)
 \exp (-\pi  |w|\sin \theta)
-\exp (-\pi i|w|\cos \theta)
 \exp ( \pi  |w|\sin \theta).
\end{eqnarray}
By the definition of $r_{\rm M}$ (in Eq. (\ref{eqn:ConRM}) for Definition \ref{dfi:intr}), the denominator $I^{({\rm De})}$ in (above) Eq. (\ref{eqn:DeA1}) for large $|w|$ is approximated by
\begin{eqnarray}
\label{eqn:DeA21}
I^{({\rm De})} \approx
-\exp (-\pi i|w|\cos \theta)
 \exp ( \pi  |w|\sin \theta)
\hspace{2ex}
\mbox{ for large }|w| \mbox{ with } (|w| \geq r_{\rm M})
\mbox{ and } \sin \theta >0
\hspace{2ex}
(\theta=\frac{3}{4}\pi,\frac{1}{4}\pi),
\end{eqnarray}
whereas
\begin{eqnarray}
\label{eqn:DeA22}
I^{\rm (De)} \approx
\exp ( \pi i|w|\cos \theta)
\exp (-\pi  |w|\sin \theta)
\hspace{2ex}
\mbox{ for large }|w| \mbox{ with } (|w| \geq r_{\rm M})
\mbox{ and } \sin \theta <0
\hspace{2ex}
(\theta=\frac{-1}{4}\pi,\frac{-3}{4}\pi).
\end{eqnarray}
Then, (above) Eqs. (\ref{eqn:DeA21})-(\ref{eqn:DeA22}) are reduced to
\begin{eqnarray}
\label{eqn:DeA3}
|I^{\rm (De)}| \approx
\exp (\pi  |w||\sin \theta|)
\hspace{2ex}
\mbox{ for large }|w| \mbox{ with } (|w| \geq r_{\rm M})
\mbox{ and }
\theta=\frac{\pm 1}{4}\pi,\frac{\pm 3}{4}\pi.
\end{eqnarray}
\par
Accordingly, combining Eqs. (\ref{eqn:Ivn}), (\ref{eqn:NuAB}) and (\ref{eqn:DeA3}), the absolute value of the integrands in Eq. (\ref{eqn:Shrl}) is
\begin{eqnarray}
\label{eqn:Ivh}
I_{\rm h}:=
\frac{|I^{(\rm P)}|
|I^{(\rm {Nu})\pm}|}{|I^{({\rm De})}|}
= |w|^{v_{\rm R}} \exp(-v_{\rm I}\theta)
  \exp (-\pi |w|^2|\sin 2\theta|)
  \exp (-\pi |w||\sin\theta|).
\end{eqnarray}
Then, 
(above) Eq. (\ref{eqn:Ivh}), for large  $|w|$, is approximated by
\begin{eqnarray}
\label{eqn:DIv}
\nonumber
I_{\rm h}
\leq 
|w|^{v_{\rm R}}
\exp(-v_{\rm I}\theta)
\exp (-\pi r_{\rm M}^2|\sin 2\theta|)
  \exp (-\pi |w||\sin\theta|)
\end{eqnarray}
\begin{eqnarray}
\nonumber
\leq
|w|^{v_{\rm R}}
\exp(-v_{\rm I}\theta)
\exp (-\pi |w||\sin \theta |)
\end{eqnarray}
\begin{eqnarray}
\approx
|w|^{v_{\rm R}}
\exp (-\pi |w||\sin \theta |)
\hspace{2ex}\mbox{ for large } |w| \geq r_{\rm M},
\end{eqnarray}
where, in the (above) last equation, the constant $\exp(-v_{\rm I}\theta)$ were disregarded.
Using Eq. (\ref{eqn:wthe}) for the straight-line contour, we have
\begin{eqnarray}
\label{eqn:thet1}
dw=d|w|\exp(i\theta)
\hspace{4ex}\mbox{ with } |\exp(i\theta)|=1.
\end{eqnarray}
Additionally, we denote the sign factor $\sigma \in \mathbb{N}$ due to the direction of integration by
\begin{eqnarray}
\label{eqn:sigm1}
\sigma:=
\left\{\begin{array}{ll}
-1 \mbox{ with } |\sigma|=1 &
\mbox{for contours such as }
({\rm \tilde{C}_{l1}}, {\rm \tilde{C}_{r1}})
\mbox{ oriented to the (coordinate) origin}\\
\mbox{ } & \mbox{ } \\
+1 &
\mbox{for contours such as }
({\rm \tilde{C}_{l7}}, {\rm \tilde{C}_{r7}})
\mbox{ oriented in the} \exp( \frac{-1}{4}\pi i)\infty, \exp( \frac{ -3}{4}\pi i)\infty
\mbox{ direction}
\end{array}\right. \hspace{-1ex}.
\end{eqnarray}
\par
Using Eqs. (\ref{eqn:Ivh})-(\ref{eqn:sigm1}) (taking into account that $-1<v_{\rm R}={\rm Re}(v)=-z_{\rm R}, z_{\rm R}-1<0$ for $v=-z, z-1$), the integrals of $I_{\rm h}$ (in Eq. \ref{eqn:DIv}) over the region $|w| \geq r_{\rm M}$ lead to
\begin{eqnarray}
\label{eqn:SCh}
\nonumber
|I^{\rm S}_{\tilde{\rm C}_{\rm lh}}|
\mbox{ and }
|I^{\rm S}_{\tilde{\rm C}_{\rm rh}}|
\leq
|\sigma| | \int_{r_{\rm M}}^{\infty} dw
I_{\rm h} |
\leq
|\sigma| | \exp(i\theta)|  \int_{r_{\rm M}}^{\infty} d|w|
I_{\rm h}
\end{eqnarray}
\begin{eqnarray}
\label{eqn:IntLp0}
\nonumber
\leq 
\int_{r_{\rm M}}^{\infty} d|w|
[ r_{\rm M}^{v_{\rm R}}
\exp (-\pi |w| |\sin \theta |) ]
<
r_{\rm M}^{v_{\rm R}}
\int_{0}^{\infty} d|w|
[\exp (-\pi |w| |\sin \theta| ) ]
\end{eqnarray}
\begin{eqnarray}
=
r_{\rm M}^{v_{\rm R}}
\frac{1}{(\pi |\sin \theta|)}
\hspace{4ex}
\mbox{ with }
v_{\rm R}=-z_{\rm R}, z_{\rm R}-1
\mbox{ and } 
\theta=\frac{\pm 1}{4}\pi,\frac{\pm 3}{4}\pi.
\end{eqnarray}
In the last integral, we
used the Laplace transform \cite{Moriguchi et al. 195619571960}.
Thus, the integral in (above) Eq. (\ref{eqn:IntLp0}) is finite. Namely, using Eqs. (\ref{eqn:Shrl}), (\ref{eqn:Angt1})-(\ref{eqn:Angt2}) and (\ref{eqn:IntLp0}), we derive
\begin{eqnarray}
\label{eqn:SMl}
I^{\rm S}_{\tilde{\rm C}_{\rm lh}}
=\mbox{ finite value (integral along either of contours }{\rm \tilde{C}_{l1}},{\rm \tilde{C}_{l4}}),
\end{eqnarray}
\begin{eqnarray}
\label{eqn:SMr}
I^{\rm S}_{\tilde{\rm C}_{\rm rh}}
=\mbox{ finite value (integral along either of contours }{\rm \tilde{C}_{r1}},{\rm \tilde{C}_{r4}}).
\end{eqnarray}
This implies that the above integrals are independent of the arc radius $r_1$ (in Definitions \ref{dfi:intr}, \ref{dfi:intw}) and negligible compared with those with singularities (in Lemmas \ref{lem:Intmr}, \ref{lem:singc}) around the (coordinate) origin in the limit $r_{1} \rightarrow \infty$.
\end{proof}

We now prove a lemma which shows that when the contours (in Figs. \ref{fig:fig1}, \ref{fig:fig2}) are in the region with intermediate distance to the origin, the integrals in the completed zeta function are finite as well.
\par
\begin{lemm}
\label{lem:IntMm}
Similarly with Lemma \ref{lem:IntLM}, 
let $w,v, z\in \mathbb{C}$ (with $v=-z, z-1$), and let $z_{\rm R}={\rm Re}(z)$ with $0<z_{\rm R}<1$.
Let $r_{\rm m}$ and $r_{\rm M} \in \mathbb{R}$ (with $r_{\rm m} < r_{\rm M}$) be the small and large radii (bounds of contours as in Definitions \ref{dfi:intr}, \ref{dfi:intw}) of $w$ along the (shifted straight-line) contours ${\rm \tilde{C}_{lm}}$ and ${\rm \tilde{C}_{rm}}$, where ${\rm \tilde{C}_{lm}}$ is either of the contours denoted by ${\rm \tilde{C}_{l2}} and {\rm \tilde{C}_{l3}}$ (in Fig. \ref{fig:fig1}), while ${\rm \tilde{C}_{rm}}$ is either of the contours denoted by ${\rm \tilde{C}_{r2}} and {\rm \tilde{C}_{r3}}$ (in Fig. \ref{fig:fig2}).
Then, the following integrals of the integrands in Eqs. (\ref{eqn:zetl})-(\ref{eqn:zet0})
\begin{eqnarray}
\label{eqn:SMmm}
\nonumber
I^{\rm S}_{\tilde{\rm C}_{\rm lm}}
    =\int_{\tilde{\rm C}_{\rm lm}} dw
\frac{w^{-z}\exp(-\pi i w^2)}{\exp(\pi iw)-\exp(-\pi iw)},
\hspace{4ex}
I^{\rm S}_{\tilde{\rm C}_{\rm rm}}
    =\int_{\tilde{\rm C}_{\rm rm}} dw
\frac{w^{z-1}\exp(+\pi i w^2)}{\exp(\pi iw)-\exp(-\pi iw)}
\end{eqnarray}
\begin{eqnarray}
\hspace{20ex}
\mbox{ along the
contours }
\tilde{\rm C}_{\rm lm}=\tilde{\rm C}_{\rm l2}, \tilde{\rm C}_{\rm l3}
\mbox{ (in Fig. \ref{fig:fig1}) and }
\tilde{\rm C}_{\rm rm}=\tilde{\rm C}_{\rm r2}, \tilde{\rm C}_{\rm r3}
\mbox{ (in Fig. \ref{fig:fig2})},
\end{eqnarray}
are finite (negligible compared with those with singularities around the origin of coordinates).
\end{lemm}

\begin{proof}
The denominator $I^{({\rm De})}$ of the integrands in Eq. (\ref{eqn:SMmm}) is rewritten as
\begin{eqnarray}
\label{eqn:DENen}
I^{({\rm De})}=\exp(\pi iw)-\exp(-\pi iw)
=\exp(\pi iw)[1-\exp(-2\pi iw)].
\end{eqnarray}
We further denote the parts of the above denominator (in Eq. (\ref{eqn:DENen})) by
\begin{eqnarray}
\label{eqn:InpD1}
I^{({\rm De})a}=\exp(\pi iw),
\end{eqnarray}
\begin{eqnarray}
\label{eqn:InpD2}
I^{({\rm De})b}=1-\exp(-2\pi iw).
\end{eqnarray}
Here, let $\theta \in \mathbb{R}$ be the angle (argument measured counterclockwise from the real axis in the complex $w$-plane), then
\begin{eqnarray}
\label{eqn:Angm1}
\theta= \frac{3}{4}\pi \hspace{2ex}\mbox{for contour ${\rm \tilde{C}_{l2}}$},
\hspace{4ex}
\theta=\frac{-1}{4}\pi \hspace{2ex}\mbox{for contour ${\rm \tilde{C}_{l3}}$},
\end{eqnarray}
\begin{eqnarray}
\label{eqn:Angm2}
\theta= \frac{1}{4}\pi \hspace{2ex}\mbox{for contour ${\rm \tilde{C}_{r2}}$},
\hspace{4ex}
\theta=\frac{-3}{4}\pi \hspace{2ex}\mbox{for contour ${\rm \tilde{C}_{r3}}$}.
\end{eqnarray}
Using $w=|w|(\cos\theta+i\sin\theta)$ (in Eq. (\ref{eqn:wthe})) and Eq. (\ref{eqn:InpD1}), it follows that
\begin{eqnarray}
\nonumber
I^{({\rm De})a}=\exp[\pi i|w|(\cos\theta+i\sin\theta)]
\end{eqnarray}
\begin{eqnarray}
=\exp (\pi i|w|\cos\theta) \exp(-\pi|w|\sin\theta),
\end{eqnarray}
yielding
\begin{eqnarray}
|I^{({\rm De})a}|=\exp(-\pi|w|\sin\theta).
\end{eqnarray}
\par
Meanwhile, from Eq. (\ref{eqn:InpD2}) (with Eq. (\ref{eqn:wthe})), we derive
\begin{eqnarray}
\label{eqn:DENb2}
I^{({\rm De})b}=1-\exp (-2\pi i|w|\cos\theta)\exp(2\pi|w|\sin\theta).
\end{eqnarray}
For $\sin\theta>0$ and $r_{\rm m} \leq |w| \leq r_{\rm M}$ ($r_{\rm m}$ and $r_{\rm M}$ are the radii defined in Definitions \ref{dfi:intr}, \ref{dfi:intw} with Figs. \ref{fig:fig1}, \ref{fig:fig2} for the contours in Eq. (\ref{eqn:SMmm})), the following quantity in the second term on the right-hand side of above Eq. (\ref{eqn:DENb2}) is larger than unity (one), that is,
\begin{eqnarray}
\label{eqn:DENb22}
\exp(2\pi|w|\sin\theta) \geq \exp(2\pi r_{\rm m}\sin\theta)>1
\hspace{2ex}
\mbox{ for } \sin\theta>0. 
\end{eqnarray}
Furthermore, the second term on the right-hand side of Eq. (\ref{eqn:DENb2}) is a complex number with radius (modulus) 
denoted as
$\exp(2\pi|w|\sin\theta)$ 
and angle (argument) $-2\pi |w|\cos\theta$, whose distance to the point 1=(1,0) is equal to $|I^{({\rm De})b}|$. This distance $|I^{({\rm De})b}|$ is larger than the difference between the above radius $\exp(2\pi|w|\sin\theta)$ and the radius of the unit circle (centered at the orogin of coordinates), namely,
\begin{eqnarray}
\label{eqn:Demin0}
|I^{({\rm De})b}|=|1-\exp (-2\pi i|w|\cos\theta)\exp(2\pi|w|\sin\theta)|
 \geq \exp(2\pi|w|\sin\theta)-1>0
\hspace{2ex}
\mbox{ for } \sin\theta>0. 
\end{eqnarray}
Combining Eqs. (\ref{eqn:DENb22}) and (\ref{eqn:Demin0}), we have (taking into account that $r_{\rm m} \leq |w| \leq r_{\rm M}$)
\begin{eqnarray}
\label{eqn:Demin}
|I^{({\rm De})b}| \geq \exp(2\pi|w|\sin\theta)-1
\geq \exp(2\pi r_{\rm m}\sin\theta)-1> 0
\hspace{2ex}\mbox{ for } \sin\theta > 0.
\end{eqnarray}
Similarly, for $\sin\theta<0$ and $r_{\rm m} \leq |w| \leq r_{\rm M}$, we obtain the following relation, corresponding to Eq. (\ref{eqn:DENb22}),
\begin{eqnarray}
\label{eqn:DENb3}
1>\exp(2\pi|w|\sin\theta) \geq \exp(2\pi r_{\rm m}\sin\theta)>0
\hspace{2ex}
\mbox{ for } \sin\theta<0. 
\end{eqnarray}
The distance $|I^{({\rm De})b}|$ between the second term on the right in Eq. (\ref{eqn:DENb2}) and the point 1=(1,0) in this case is larger than the difference between the aforementioned radius (modulus) 
$\exp(2\pi|w|\sin\theta)$ 
and the radius of the unit circle (centered at the origin of coordinates). We then have (considering $r_{\rm m} \leq |w| \leq r_{\rm M}$) that
\begin{eqnarray}
\label{eqn:Demot}
|I^{({\rm De})b}| \geq 1-\exp(2\pi|w|\sin\theta)
\geq 1-\exp(2\pi r_{\rm m}\sin\theta)> 0
\hspace{2ex}\mbox{ for } \sin\theta <0.
\end{eqnarray}
\par
In contrast, using the notation $v=-z, z-1$, the parts of the numerators of the integrands in Eq. (\ref{eqn:SMmm}) can be written as
\begin{eqnarray}
\label{eqn:Pom}
I^{(\rm Po)}_v=w^v
\mbox{ with } v=-z, z-1,
\end{eqnarray}
\begin{eqnarray}
\label{eqn:Num}
I^{({\rm Nu})\mp}=\exp(\mp\pi i w^2).
\end{eqnarray}
By denoting $w=|w|(\cos\theta+i\sin\theta)$, we obtain the same results as those in Eqs. (\ref{eqn:PwAb}) and (\ref{eqn:NuAB}) (in Lemma \ref{lem:IntMm}). Namely,
\begin{eqnarray}
\label{eqn:PoABm}
|I^{(\rm Po)}_v|=|w|^{v_{\rm R}} \exp(-v_{\rm I}\theta),
\end{eqnarray}
\begin{eqnarray}
\label{eqn:NuABm}
|I^{{\rm Nu})\mp}|
=\exp (-\pi |w|^2|\sin 2\theta| )
\hspace{2ex}
\mbox{ with }
\theta=\frac{\pm 1}{4}\pi,\frac{\pm 3}{4}\pi.
\end{eqnarray}
Therefore, combining
Eq. (\ref{eqn:Demin}) (or Eq. (\ref{eqn:Demot})) and Eqs. (\ref{eqn:PoABm})-(\ref{eqn:NuABm}), we obtain that the absolute value of the integrands is
\begin{eqnarray}
\label{eqn:Ivm}
I_{\rm m}:=
\frac{|I^{(\rm P)}||I^{(\rm {Nu})\mp}|}{|I^{({\rm De})}|}.
\end{eqnarray}
\par
Using Eqs. (\ref{eqn:Demin}) (or Eq. (\ref{eqn:Demot})) and Eqs. (\ref{eqn:PoABm})-(\ref{eqn:Ivm}) (with consideration that $-1<v_{\rm R}={\rm Re}(v)=-z_{\rm R}, z_{\rm R}-1<0$ for $v=-z, z-1$, as well as $|\sigma \exp(i\theta)|=1 $ in Eqs. (\ref{eqn:thet1})-(\ref{eqn:SCh})), we obtain
\begin{eqnarray}
\label{eqn:Intm1}
\nonumber
|I^{\rm S}_{\tilde{\rm C}_{\rm lm}}|
\mbox{ and }
|I^{\rm S}_{\tilde{\rm C}_{\rm rm}}|
\leq
|\int_{r_{\rm m}}^{r_{\rm M}}  d|w| I_{\rm m}|
\end{eqnarray}
\begin{eqnarray}
\leq
(r_{\rm M}-r_{\rm m})
\frac{
r_{\rm m}^{v_{\rm R}} \exp(-v_{\rm I}\theta)
\exp (-\pi r_{\rm m}^2|\sin 2\theta|)
}
{|\exp(2\pi r_{\rm m}\sin\theta)-1|}
\hspace{2ex}
\mbox{ with } \theta=\frac{\pm 1}{4}\pi, \theta=\frac{\pm 3}{4}\pi.
\end{eqnarray}
Hence, the above straight-line integrals (in Eq. (\ref{eqn:Intm1}) of $I_{\rm m}$ (in Eq. (\ref{eqn:Ivm})) with respect to $|w|$ in the region $r_{\rm m} \leq |w| \leq w_{\rm M}$ are smaller than the terms proportional to $r_{\rm m}^{v_{\rm R}}$ (disregarding the multiplied constants) with $-1<v_{\rm R}<0$, and take finite values. This finiteness is due to the large value of the radius $r_{\rm m}$, which is independent of the radius $r_{1}$ of the arc contours (in Figs. \ref{fig:fig1}, \ref{fig:fig2}) around the (coordinate) origin with $r_{\rm m} >> r_{1}$ (as in Definitions \ref{dfi:intr}, \ref{dfi:intw}). Therefore, singularities do not occur here unlike the case of integrals (in Lemmas \ref{lem:Intmr}, \ref{lem:singc}) around the (coordinate) origin in the limit of $r_1 \rightarrow 0$.
Namely,
\begin{eqnarray}
\label{eqn:Sml}
I^{\rm S}_{\tilde{\rm C}_{\rm lm}}
=\mbox{ finite value (integral along either of contours }{\rm \tilde{C}_{l2}},{\rm \tilde{C}_{l3}}\mbox{ in Fig. \ref{fig:fig1})},
\end{eqnarray}
\begin{eqnarray}
\label{eqn:Smr}
I^{\rm S}_{\tilde{\rm C}_{\rm rm}}
=\mbox{ finite value (integral along either of contours }{\rm \tilde{C}_{r2}},{\rm \tilde{C}_{r3}} \mbox{ in Fig. \ref{fig:fig2})}.
\end{eqnarray}
This implies that the above integrals are negligible compared with those with singularities around the (coordinate) origin.
\end{proof}
\par
In this section, we deformed and shifted the contours denoted by $0 \searrow 1$ and $0 \swarrow 1$ in the integral form of the (completed) zeta function given by Eqs. (\ref{eqn:zetl})-(\ref{eqn:zet0}) to those shown in Fig. \ref{fig:fig1} and Fig. \ref{fig:fig2}, respectively (in Theorem \ref{lem:Defm}). Then, we separated the finite integrals (in Lemmas \ref{lem:IntLM}, \ref{lem:IntMm}) along the shifted straight-line contours from the integral around the (coordinate) origin containing the singularities, which appear only when the contours approach the origin.

\section{
Proof of the 
Riemann's
conjecture that
the real part of
all non-trivial
zeros of the zeta function
is 1/2
}
\label{sec:3}

In previous Section \ref{sec:2} (with Lemmas \ref{lem:IntLM}, \ref{lem:IntMm}), it was shown that integrals of the integrands in $\zeta_{\rm l}$ and $\zeta_{\rm r}$ (in Eqs. (\ref{eqn:zetl})-(\ref{eqn:zetr})) for the integral form of the completed zeta function in Eq. (\ref{eqn:zet0})), along the shifted straight-line contours (in Figs. \ref{fig:fig1}, \ref{fig:fig2}), which are away from the (coordinate) origin (in the complex plane), are always finite and do not have  singularities.
(Note: Using Theorem \ref{thm:czta} and Eq. (\ref{eqn:zets}), the complex numbers $z$ and $1-z$ for $\zeta_{\rm l}$ and $\zeta_{\rm r}$ in the completed zeta function in Eq. (\ref{eqn:zeta0}) were exchanged as shown in Eqs. (\ref{eqn:zetl})-(\ref{eqn:zet0})).
In this section, we show that the integrals along the contours near the (coordinate) origin (in Figs. \ref{fig:fig1}, \ref{fig:fig2}) have singularities (in Lemmas \ref{lem:Intmr}, \ref{lem:singc}) when the radius of the arc contours approaches zero.
Then,
in Theorem \ref{thm:ReH}, we prove that
the real part of
all non-trivial
zeros of the zeta function
must be 1/2.
As it is known that
all
non-trivial zeros of the zeta function exist in the region $0< {\rm Re } (z) <1$ in literature \cite{Poussin 1896,Hadamard 1896}, we concentrate on this region. Furthermore, it is also known that the number of zeros (of the zeta function) with a real part of 1/2 is infinite \cite{Hardy 1914}.
To derive the real part of non-trivial zeros of the zeta function, the present approach uses (in addition to the above symmetry given by Eq. (\ref{eqn:zets})) the property (with merits) that a quantity in one term in a highly (attainable) symmetrized integral form generates a corresponding (paired) quantity in another term.
\par
Similarly with Lemmas \ref{lem:IntLM}, \ref{lem:IntMm}, we here evaluate the integrals of the form in Eqs. (\ref{eqn:zetl})-(\ref{eqn:zet0}), and separate singularities.

\begin{lemm}
\label{lem:Intmr}
Let $w,v, z\in \mathbb{C}$ (with $v=-z, z-1$), and let $z_{\rm R}={\rm Re}(z)$ with $0<z_{\rm R}<1$.
Let $r_1$ and $r_{\rm m} \in \mathbb{R}$ with $r_1 << r_{\rm m}$ be the small radii (bounds of contours as in Definitions \ref{dfi:intr}, \ref{dfi:intw}) of $w$ along the (shifted straight-line) contours ${\rm \tilde{C}_{ls}}$ and ${\rm \tilde{C}_{rs}}$, where ${\rm \tilde{C}_{ls}}$ is either of the contours denoted by ${\rm \tilde{C}_{lp}},{\rm \tilde{C}_{ln}}$ (in Fig. \ref{fig:fig1}), while ${\rm \tilde{C}_{rs}}$ is  either of the contours denoted by ${\rm \tilde{C}_{rp}}, {\rm \tilde{C}_{rn}}$ (in Fig. \ref{fig:fig2}).
Then, the following integrals of the integrands in Eqs. (\ref{eqn:zetl})-(\ref{eqn:zet0})
\begin{eqnarray}
\label{eqn:Smrs}
\nonumber
I^{\rm S}_{\tilde{\rm C}_{\rm ls}}
    =\int_{\tilde{\rm C}_{\rm ls}} dw
\frac{w^{-z}\exp(-\pi i w^2)}{\exp(\pi iw)-\exp(-\pi iw)},
\hspace{4ex}
I^{\rm S}_{\tilde{\rm C}_{\rm rs}}
    =\int_{\tilde{\rm C}_{\rm rs}} dw
\frac{w^{z-1}\exp(+\pi i w^2)}{\exp(\pi iw)-\exp(-\pi iw)}
\end{eqnarray}
\begin{eqnarray}
\hspace{20ex}
\mbox{ along the
contours }
\tilde{\rm C}_{\rm ls}=\tilde{\rm C}_{\rm lp}, \tilde{\rm C}_{\rm ln}
\mbox{ (in Fig. \ref{fig:fig1}) and }
\tilde{\rm C}_{\rm rs}=\tilde{\rm C}_{\rm rp}, \tilde{\rm C}_{\rm rn}
\mbox{ (in Fig. \ref{fig:fig2})},
\end{eqnarray}
have singularities in the limit of $r_1 \rightarrow 0$. The power of singularities of these integrals is $r_1^{-z_{\rm R}}$ on the left in above Eq. (\ref{eqn:Smrs}) (for the contours $\tilde{\rm C}_{\rm ls}$ $=\tilde{\rm C}_{\rm lp}, \tilde{\rm C}_{\rm ln}$), while the corresponding power is $r_1^{z_{\rm R}-1}$  on the right in above Eq. (\ref{eqn:Smrs}) (for the contours $\tilde{\rm C}_{\rm rs} =\tilde{\rm C}_{\rm rp}, \tilde{\rm C}_{\rm rn}$).
\end{lemm}

\begin{proof}
For $0< r_1 < |w| < r_{\rm m}$ and $r_1 << r_{\rm m}$ with $r_1$ and $r_{\rm m}$ being the small radii (bounds of contours as in Definitions \ref{dfi:intr}, \ref{dfi:intw}) along the straight-line contours ${\rm \tilde{C}_{l2}}, {\rm \tilde{C}_{l3}}$, ${\rm \tilde{C}_{r2}}, {\rm \tilde{C}_{r3}}$ (in Figs. \ref{fig:fig1}, \ref{fig:fig2}),
the denominator $I^{({\rm De})}$ and parts of the numerators, $I^{({\rm Nu})-}$ and $I^{({\rm Nu})+}$ in Eq. (\ref{eqn:Smrs}) are approximated by
\begin{eqnarray}
\label{eqn:Desa}
I^{({\rm De})}=\exp(\pi iw)-\exp(-\pi iw)
\approx 2\pi i w,
\end{eqnarray}
\begin{eqnarray}
\label{eqn:NusNP}
I^{({\rm Nu})-}=\exp(-\pi i w^2) \approx 1,
\hspace{4ex}
I^{({\rm Nu})+}=\exp(+\pi i w^2) \approx 1.
\end{eqnarray}
The polynomials in above Eq. (\ref{eqn:Smrs}) can be written as
\begin{eqnarray}
\label{eqn:Pos}
I^{(\rm Po)}=w^v
\mbox{ with } v=-z, z-1.
\end{eqnarray}
Then, using Eqs. (\ref{eqn:Poh}) and Eq. (\ref{eqn:wthe}) with $\theta$ (angle along the straight-line contours measured counterclockwise from the real axis in the complex $w$-plane), above $I^{(\rm Po)}$ (in Eq. (\ref{eqn:Pos})) can be expressed as
\begin{eqnarray}
\label{eqn:Pose}
\nonumber
I^{(\rm Po)}=w^{v_{\rm R}-1}w^{iv_{\rm I}}
\end{eqnarray}
\begin{eqnarray}
\nonumber
=w^{v_{\rm R}-1}
\exp[iv_{\rm I} \ln(|w|)]\exp[-v_{\rm I}\arg(w)]
\end{eqnarray}
\begin{eqnarray}
=|w|^{v_{\rm R}-1}\exp[{i(v_{\rm R}-1})\theta]
\exp[iv_{\rm I} \ln(|w|)]\exp(-v_{\rm I}\theta),
\end{eqnarray}
(with $v_{\rm R}={\rm Re}(v)$, $v_{\rm I}={\rm Im}(v)$).
The angle $\theta$ in this case is denoted by
\begin{eqnarray}
\label{eqn:Angmr}
\theta= \frac{3}{4}\pi \hspace{2ex}\mbox{for contour ${\rm \tilde{C}_{lp}}$},
\hspace{4ex}
\theta=\frac{-1}{4}\pi \hspace{2ex}\mbox{for contour ${\rm \tilde{C}_{ln}}$},
\end{eqnarray}
\begin{eqnarray}
\label{eqn:Angtr2}
\theta= \frac{1}{4}\pi \hspace{2ex}\mbox{for contour ${\rm \tilde{C}_{rp}}$},
\hspace{4ex}
\theta=\frac{-3}{4}\pi \hspace{2ex}\mbox{for contour ${\rm \tilde{C}_{rn}}$}.
\end{eqnarray}
From Eqs. (\ref{eqn:Desa})-(\ref{eqn:NusNP}) and (\ref{eqn:Pose}), we obtain 
(disregarding $2\pi i$ in Eq. (\ref{eqn:Desa}) as well as the constants 
$\exp[i(v_{\rm R}-1)\theta]$ and $\exp(-v_{\rm I}\theta)$ 
in Eq. (\ref{eqn:Pose})) 
that
\begin{eqnarray}
I_{\rm s}:=\frac{I^{(\rm Po)}  I^{({\rm Nu}) \mp} }{ I^{(\rm De)} }
= |w|^{v_{\rm R}-1} \exp[iv_{\rm I}\ln(|w|)].
\end{eqnarray}
Further disregarding $\exp(i\theta)$ in Eq. (\ref{eqn:thet1}) and the sign $\sigma$ (in Eq. (\ref{eqn:sigm1})) due to the direction of integration, we have, for the integrals in Eq. (\ref{eqn:Smrs}), that
\begin{eqnarray}
\nonumber
\label{eqn:InLI00}
\hspace{-50ex}
I^{\rm S}_{\tilde{\rm C}_{\rm ls}} \approx
I^{\rm S}_{\tilde{\rm C}_{\rm rs}}
=\int_{r_{1}}^{r_{\rm m}} d|w| |w|^{v_{\rm R}-1} \exp[iv_{\rm I}\ln(|w|)]
\end{eqnarray}
\begin{eqnarray}
\hspace{10ex}
\mbox{ along the contours } \tilde{\rm C}_{\rm ls}=
\tilde{\rm C}_{\rm lp}, \tilde{\rm C}_{\rm ln}
\mbox{ and }
\tilde{\rm C}_{\rm rs}=
\tilde{\rm C}_{\rm rp}, \tilde{\rm C}_{\rm rn}
\mbox{ in the regions } r_{\rm 1} \leq |w| \leq r_{\rm m}
\mbox{ in Figs. \ref{fig:fig1} and \ref{fig:fig2}}.
\end{eqnarray}
\par
Using the small (bounds of) radii $r_1$ and $r_{\rm m}$ of $w$ (with $0< r_1  <|w|< r_{\rm m}$ and $r_1  << r_{\rm m}$ in Definitions \ref{dfi:intr}, \ref{dfi:intw} and Figs. \ref{fig:fig1}, \ref{fig:fig2}), we introduce the parameter variables $\tilde{t}$, $\tilde{t}_1$ and $\tilde{t}_{\rm m}$ as follows:
\begin{eqnarray}
\label{eqn:InL03}
\tilde{t}:=\ln(|w|),
\hspace{4ex}
\tilde{t}_1:=\ln(r_1),
\hspace{4ex}
\tilde{t}_{\rm m}:=\ln(r_{\rm m}).
\end{eqnarray}
Then, we have
\begin{eqnarray}
\label{eqn:InL04}
|w|=\exp(\tilde{t}),
\hspace{4ex}
r_1=\exp(\tilde{t}_1),
\hspace{4ex}
r_{\rm m}=\exp(\tilde{t}_{\rm m}),
\end{eqnarray}
yielding
\begin{eqnarray}
\label{eqn:InLI05}
d|w|=d{\tilde{t}}[\exp(\tilde{t})].
\end{eqnarray}
\par
By using Eqs. (\ref{eqn:InLI00})-(\ref{eqn:InLI05}), we get
\begin{eqnarray}
\label{eqn:InLI07}
\nonumber
I^{\rm S}_{\tilde{\rm C}_{\rm ls}}
\approx
I^{\rm S}_{\tilde{\rm C}_{\rm rs}}=
\int_{\tilde{t}_1}^{\tilde{t}_{\rm m}} d{\tilde{t}} [\exp(\tilde{t})]
\{ \exp[(v_{\rm R}-1)\tilde{t}] \}[\exp(iv_{\rm I}\tilde{t})]
\end{eqnarray}
\begin{eqnarray}
\nonumber
=\int_{\tilde{t}_1}^{\tilde{t}_{\rm m}} d{\tilde{t}}
 \exp[(1+v_{\rm R}-1+iv_{\rm I})\tilde{t}]
=\int_{\tilde{t}_1}^{\tilde{t}_{\rm m}} d{\tilde{t}}
 \exp[(v_{\rm R}+iv_{\rm I})\tilde{t}]
\end{eqnarray}
\begin{eqnarray}
\nonumber
=\frac{ \exp[(v_{\rm R}+iv_{\rm I})\tilde{t}_{\rm m}]
       -\exp[(v_{\rm R}+iv_{\rm I})\tilde{t}_1] }
{v_{\rm R}+iv_{\rm I}}
\end{eqnarray}
\begin{eqnarray}
=\frac{ r_{\rm m}^{v_{\rm R}+iv_{\rm I}}
       -r_1      ^{v_{\rm R}+iv_{\rm I}} }
{v_{\rm R}+iv_{\rm I}}.
\end{eqnarray}
Hence, (above) Eq. (\ref{eqn:InLI07}) becomes
\begin{eqnarray}
I^{\rm S}_{\tilde{\rm C}_{\rm ls}} \approx
I^{\rm S}_{\tilde{\rm C}_{\rm rs}}=
\frac{-r_1^{v_{\rm R}+iv_{\rm I}} }
{v_{\rm R}+iv_{\rm I}}
\hspace{4ex}\mbox{ for } 
0<r_1 << r_{\rm m}
\mbox{ in the limit of } 
r_1 \rightarrow 0,
\end{eqnarray}
yielding
\begin{eqnarray}
\label{eqn:lrpn}
\nonumber
I^{\rm S}_{\tilde{\rm C}_{\rm ls}} \approx
I^{\rm S}_{\tilde{\rm C}_{\rm rs}}
= \frac{
       -  r_1^{ v_{\rm R}}
         r_1^{iv_{\rm I}} }
     {v_{\rm R}+iv_{\rm I}}
= \frac{
      -  r_1^{ v_{\rm R}}
         \exp[\ln(r_1^{iv_{\rm I}})] }
     {v_{\rm R}+iv_{\rm I}}
= \frac{
      -  r_1^{ v_{\rm R}}
         \exp\{iv_{\rm I}[\ln(|r_1|)+i\arg(r_1)]\} }
     {v_{\rm R}+iv_{\rm I}}
\end{eqnarray}
\begin{eqnarray}
= \frac{
      - r_1^{v_{\rm R}}
      \exp[iv_{\rm I}\ln(|r_1|)- v_{\rm I}\arg(r_1)]
}
{v_{\rm R}+iv_{\rm I}}.
\end{eqnarray}
By dropping the constant containing
$v_{\rm I}\arg(r_1)$,
the absolute value of (above) Eq. (\ref{eqn:lrpn}) is reduced to
\begin{eqnarray}
\label{eqn:Imr}
|I^{\rm S}_{\tilde{\rm C}_{\rm ls}}|
\approx
|I^{\rm S}_{\tilde{\rm C}_{\rm rs}}|=
 \frac{ r_1 ^{v_{\rm R}}
}
{|v_{\rm R}+iv_{\rm I}|}.
\end{eqnarray}
\par
Thus, disregarding the constant $|v_{\rm R}+iv_{\rm I}|$
in (above) Eq. (\ref{eqn:Imr}), we derive, in the limit of $r_1 \rightarrow 0$ (with $v=-z, z-1, v_{\rm R}={\rm Re}(v), 0<z_{\rm R}={\rm Re}(z) <1$), that
\begin{eqnarray}
\label{eqn:Sml}
|I^{\rm S}_{\tilde{\rm C}_{\rm ls}}|
\approx
r_1^{-z_{\rm R}}
\mbox{ (along the contours } \tilde{\rm C}_{\rm ls}=
\tilde{\rm C}_{\rm lp}, \tilde{\rm C}_{\rm ln}
\mbox{ in Fig. \ref{fig:fig1}) },
\end{eqnarray}
\begin{eqnarray}
\label{eqn:Smr}
|I^{\rm S}_{\tilde{\rm C}_{\rm rs}}|
\approx
r_1^{z_{\rm R}-1}
\mbox{ (along the contours } \tilde{\rm C}_{\rm rs}=
\tilde{\rm C}_{\rm rp}, \tilde{\rm C}_{\rm rn}
\mbox{ in Fig. \ref{fig:fig2})},
\end{eqnarray}
implying that the power of singularities of these integrals is $r_1^{-z_{\rm R}}$ on the left-hand side in Eq. (\ref{eqn:Smrs}), whereas the corresponding power is $r_1^{z_{\rm R}-1}$  on the right-hand side in Eq. (\ref{eqn:Smrs}).
\end{proof}

We now evaluate the circular integrals along the arc contours around the (coordinate) origin in Figs. \ref{fig:fig1}, \ref{fig:fig2}. These integrals have singularities when the radius of the arc approaches zero.

\begin{lemm}
\label{lem:singc}
Let $w,v, z\in \mathbb{C}$ (with $v=-z, z-1$), and let $z_{\rm R}={\rm Re}(z)$ with $0<z_{\rm R}<1$.
Let ${\rm \tilde{C}_{lc}}$ and ${\rm \tilde{C}_{rc}}$ be the deformed-arc contours around the (coordinate) origin in Figs. \ref{fig:fig1}, \ref{fig:fig2}.
Let $r_1 \in \mathbb{R}$ be the small radius (in Definitions \ref{dfi:intr}, \ref{dfi:intw}) of the above contours ${\rm \tilde{C}_{lc}}$ and ${\rm \tilde{C}_{rc}}$.
Then, the following circular integrals of the integrands in Eqs. (\ref{eqn:zetl})-(\ref{eqn:zet0})
\begin{eqnarray}
\label{eqn:Smrc}
\nonumber
I^{\rm S}_{\tilde{\rm C}_{\rm lc}}
    =\int_{\tilde{\rm C}_{\rm lc}} dw
\frac{w^{-z}\exp(-\pi i w^2)}{\exp(\pi iw)-\exp(-\pi iw)},
\hspace{4ex}
I^{\rm S}_{\tilde{\rm C}_{\rm rc}}
    =\int_{\tilde{\rm C}_{\rm rc}} dw
\frac{w^{z-1}\exp(+\pi i w^2)}{\exp(\pi iw)-\exp(-\pi iw)}
\end{eqnarray}
\begin{eqnarray}
\hspace{30ex}
\mbox{ along either of arc contours }
\tilde{\rm C}_{\rm lc}
\mbox{ (in Fig. \ref{fig:fig1}) and }
\tilde{\rm C}_{\rm rc}
\mbox{ (in Fig. \ref{fig:fig2})},
\end{eqnarray}
have singularities in the limit of $r_1 \rightarrow 0$. The powers of these singularities are $|r_1|^{-z_{\rm R}}$ and $|r_1|^{z_{\rm R}-1}$ for the contours $\tilde{\rm C}_{\rm lc}$ and $\tilde{\rm C}_{\rm rc}$ (on the left and right in Eq.(\ref{eqn:Smrc})), respectively.
\end{lemm}
\par
\begin{proof}
For the small radius $r_1$ (in Definition \ref{dfi:intr}) of the deformed-arc contours (${\rm \tilde{C}_{lc}}$ and ${\rm \tilde{C}_{rc}}$ in Figs. \ref{fig:fig1}, \ref{fig:fig2}),
the denominator $I^{({\rm De})}$ and parts of the numerators, $I^{({\rm Nu})-}$ and $I^{({\rm Nu})+}$ in Eq. (\ref{eqn:Smrc}) are approximated by
\begin{eqnarray}
\label{eqn:Deca}
I^{({\rm De})}=\exp(\pi iw)-\exp(-\pi iw)
\approx 2\pi i w,
\end{eqnarray}
\begin{eqnarray}
\label{eqn:NucNP}
I^{({\rm Nu})-}=\exp(-\pi i w^2) \approx 1,
\hspace{4ex}
I^{({\rm Nu})+}=\exp(+\pi i w^2) \approx 1.
\end{eqnarray}
The polynomials in above Eq. (\ref{eqn:Smrc}) can be written as
\begin{eqnarray}
\label{eqn:Poce}
I^{(\rm Po)}=w^v
\mbox{ with } v=-z, z-1.
\end{eqnarray}
From Eqs. (\ref{eqn:Deca})-(\ref{eqn:Poce}), we obtain (disregarding $2\pi i$ in Eq. (\ref{eqn:Deca}))
\begin{eqnarray}
\label{eqn:Icv}
I_{{\rm c},v}:=\frac{I^{(\rm Po)} I^{({\rm Nu}) \mp} }{ I^{(\rm De)} }
= w^{v-1}
\mbox{ with } v=-z, z-1.
\end{eqnarray}
\par
Using (above) Eq. (\ref{eqn:Icv}), the integral along the arc contour ${\rm \tilde{C}_{\rm lc}}$ (in Fig.\ref{fig:fig1}) on the left-hand side in Eq. (\ref{eqn:Smrc}) can be expressed as
\begin{eqnarray}               
\label{eqn:Ivl}
I^{\rm S}_{\rm \tilde{C}_{\rm lc}}
=\int_{\rm \tilde{C}_{lc}} dw
I_{{\rm c},v=-z}
\int_{\rm \tilde{C}_{lc}} dw (w^{-z-1}).
\end{eqnarray}
Let $\phi_{\rm c}$ be the angle (argument) along the arc measured
counterclockwise from the real axis in the complex plane.
Using (above) Eq. (\ref{eqn:Ivl}), with consideration of $|w|=r_{1}$ on the contour ${\rm \tilde{C}_{\rm lc}}$, and
\begin{eqnarray}               
\label{eqn:r1phi}
w=|w|\exp(i\phi_{\rm c})=r_1\exp(i\phi_{\rm c}),
\end{eqnarray}
with
\begin{eqnarray}               
\frac{dw}{d\phi_{\rm c}}=ir_1\exp(i\phi_{\rm c})=iw,
\end{eqnarray}
we obtain (the integral along the arc contour ${\rm \tilde{C}_{\rm lc}}$ in Fig.\ref{fig:fig1})
\begin{eqnarray}               
\label{eqn:Si1c}
I^{\rm S}_{\rm \tilde{C}_{\rm lc}}
=\int_{\frac{3}{4}\pi}^{\frac{-1}{4}\pi} 
d\phi_{\rm c}  \frac{d w}{d \phi_{\rm c}} w^{-z-1}
=\int_
{\frac{3}{4}\pi}^{\frac{-1}{4}\pi} 
d\phi_{\rm c}  (i)w^{-z}.
\end{eqnarray}
As in Eq. (\ref{eqn:Poh}),
the integrand of (above) Eq. (\ref{eqn:Si1c}) (with $z_{\rm R}={\rm Re}(z)$ and $z_{\rm I}={\rm Im}(z)$)  can be written as
\begin{eqnarray}               
\label{eqn:idlc}
\nonumber
iw^{-z}=iw^{-z_{\rm R}-iz_{\rm I}}
\end{eqnarray}
\begin{eqnarray}               
=iw^{-z_{\rm R}}\exp[-iz_{\rm I}\ln(|w|)]\exp[z_{\rm I}\arg(w)].
\end{eqnarray}               
\par
Then, 
using Eqs. (\ref{eqn:r1phi}) and (\ref{eqn:idlc}) with $\phi_{\rm c}=\arg(w)$
(for
$|w|=r_1>0$ on the contour ${\rm C_{lc}}$),
the integral in Eq. (\ref{eqn:Si1c}) becomes
\begin{eqnarray}               
\label{eqn:Il1c}
\nonumber
I^{\rm S}_{\tilde{C}_{\rm lc}}
=\int_{\frac{3}{4}\pi}^{\frac{-1}{4}\pi} d\phi_{\rm c}  (i)w^{-z_{\rm R}}\exp[-iz_{\rm I}\ln(|w|)]\exp(z_{\rm I}\phi_{\rm c})
\end{eqnarray}
\begin{eqnarray}               
\nonumber
=\int_{\frac{3}{4}\pi}^{\frac{-1}{4}\pi} d\phi_{\rm c}  (i)
|w|^{-z_{\rm R}}\exp[-iz_{\rm R}\phi_{\rm c}]
\exp[-iz_{\rm I}\ln(|w|)]\exp(z_{\rm I}\phi_{\rm c})
\end{eqnarray}
\begin{eqnarray}               
\nonumber
=i|w|^{-z_{\rm R}}\exp[-iz_{\rm I}\ln(|w|)]
\frac{
\exp[(-iz_{\rm R}+z_{\rm I})(\frac{-1}{4}\pi)]
  -\exp[(-iz_{\rm R}+z_{\rm I})(\frac{3}{4}\pi)]
}
{-iz_{\rm R}+z_{\rm I}}
\end{eqnarray}
\begin{eqnarray}               
=i|w|^{-z_{\rm R}}\exp[-iz_{\rm I}\ln(|w|)]
\frac{
\exp[(-iz_{\rm R}+z_{\rm I})(\frac{-1}{4}\pi)]
\{1-\exp[(-iz_{\rm R}+z_{\rm I})\pi] \}
}
{-iz_{\rm R}+z_{\rm I}}.
\end{eqnarray}
Therefore, the absolute value of $I^{\rm S}_{\tilde{C}_{\rm lc}}$ in (above) Eq. (\ref{eqn:Il1c}) (with $|w|=r_1$) is
\begin{eqnarray}               
\label{eqn:lSc}
|I^{\rm S}_{\tilde{C}_{\rm lc}}|
=|w|^{-z_{\rm R}}
\frac{|
\exp[(z_{\rm I})(\frac{-1}{4}\pi)]
\{1-\exp[(-iz_{\rm R}+z_{\rm I})\pi] \}
|}
{|-iz_{\rm R}+z_{\rm I}|}
\approx 
r_1^{-z_{\rm R}},
\end{eqnarray}               
which implies that, 
disregarding the constants
$\exp[(z_{\rm I})(-\pi/4)]$,
$1-\exp[(-iz_{\rm R}+z_{\rm I})\pi]$ and 
$-iz_{\rm R}+z_{\rm I}$ 
(see also Note below Eq. (\ref{eqn:srSc})), 
the integral
$|I^{\rm S}_{\tilde{C}_{\rm lc}}|$
in (above) Eq. (\ref{eqn:lSc})
has the form $|w|^{-z_{\rm R}}=r_1^{-z_{\rm R}}$ with the singularity caused by the
order
power $-z_{\rm R}$ of $|w|$ for $0<z_{\rm R}<1$
in the limit of $r_1 \rightarrow 0$.
\par
Meanwhile, the circular integral
$I^{\rm S}_{\rm \tilde{C}_{\rm rc}}$
along the arc contour ${\rm \tilde{C}_{\rm rc}}$ (in Fig. \ref{fig:fig2}) on the right in Eq. (\ref{eqn:Smrc})
with 
small $|w|=r_1$ near the origin in the complex $w$-plane is calculated by the replacement $z \rightarrow 1-z$
in Eq. (\ref{eqn:Il1c}), that is,
\begin{eqnarray}               
z_{\rm R} \rightarrow 1-z_{\rm R},
\hspace{2ex}
-z_{\rm R} \rightarrow z_{\rm R}-1,
\hspace{2ex}
z_{\rm I} \rightarrow -z_{\rm I},
\hspace{2ex}
\int_{\frac{3}{4}\pi}^{\frac{-1}{4}\pi} d\phi_{\rm c}
\rightarrow
\int_{\frac{1}{4}\pi}^{-\frac{3}{4}\pi} d\phi_{\rm c},
\end{eqnarray}
yielding (with the use of Eq. (\ref{eqn:Icv}) and $|w|=r_1$)
\begin{eqnarray}               
\label{eqn:Ir1c}
\nonumber
I^{\rm S}_{\rm \tilde{C}_{\rm rc}}
=\int_{\rm \tilde{C}_{rc}}
I_{{\rm c},v=z-1}
=\int_{\frac{1}{4}\pi}^{-\frac{3}{4}\pi} d\phi_{\rm c} (i) w^{z_{\rm R}-1}\exp[iz_{\rm I}\ln(|w|)]\exp(-z_{\rm I}\phi_{\rm c})
\end{eqnarray}
\begin{eqnarray}               
\nonumber
=\int_{\frac{1}{4}\pi}^{\frac{-3}{4}\pi} d\phi_{\rm c} (i)
|w|^{z_{\rm R}-1}\exp[i(z_{\rm R}-1)\phi_{\rm c}]
\exp[iz_{\rm I}\ln(|w|)]\exp(-z_{\rm I}\phi_{\rm c})
\end{eqnarray}
\begin{eqnarray}               
\nonumber
=i|w|^{z_{\rm R}-1}\exp[iz_{\rm I}\ln(|w|)]
\frac{
 \exp \{ [i(z_{\rm R}-1)-z_{\rm I}](\frac{-3}{4}\pi)     \}
-\exp \{ [i(z_{\rm R}-1)-z_{\rm I}](\frac{1}{4}\pi) \}
}
{i(z_{\rm R}-1)-z_{\rm I}}
\end{eqnarray}
\begin{eqnarray}               
=i|w|^{z_{\rm R}-1}\exp[iz_{\rm I}\ln(|w|)]
\frac{
 \exp \{ [i(z_{\rm R}-1)-z_{\rm I}](\frac{-3}{4}\pi) \}
\{1-\exp \{ [i(z_{\rm R}-1)-z_{\rm I}]\pi \}
}
{i(z_{\rm R}-1)-z_{\rm I}}.
\end{eqnarray}
Therefore, the absolute value of $I^{\rm S}_{\tilde{C}_{\rm rc}}$ in (above) Eq. (\ref{eqn:Ir1c}) (with $|w|=r_1$) is
\begin{eqnarray}               
\label{eqn:rSc}
|I^{\rm S}_{\rm \tilde{C}_{\rm rc}}|
=|w|^{z_{\rm R}-1}
\frac{
|
 \exp [(-z_{\rm I})(\frac{-3}{4}\pi) ]
\{
1-\exp \{[i(z_{\rm R}-1)-z_{\rm I}]\pi \}
\}
|
}
{|i(z_{\rm R}-1)-z_{\rm I}|}
\approx |r_1|^{z_{\rm R}-1},
\end{eqnarray}               
which implies that, 
disregarding the constants expressed by
$ \exp [(-z_{\rm I})(-3\pi/4) ]$
and
$
1-\exp \{[i(z_{\rm R}-1)-z_{\rm I}]\pi \}
$
as well as
$
i(z_{\rm R}-1)-z_{\rm I}
$
(see also Note below Eq.
(\ref{eqn:srSc})), 
the integral 
$|I^{\rm S}_{\rm \tilde{C}_{\rm rc}}|$ in (above) Eq. (\ref{eqn:rSc})
has the form $|w|^{z_{\rm R}-1}=r_1^{z_{\rm R}-1}$ with the singularity caused by the 
order
power $z_{\rm R}-1$ of $|w|$ for $0<z_{\rm R}<1$
in the limit of $r_1 \rightarrow 0$.
\end{proof}

\par
We then prove the following theorem, which completes the proof of the Riemann hypothesis.

\begin{theo}
\label{thm:ReH}
Let $z \in \mathbb{C}$ and let $z_{\rm R}={\rm Re}(z)$. Let $\hat{\zeta}(z)$ be the completed
zeta
function given in Theorem \ref{thm:czta}.
To satisfy $\hat{\zeta}(z)=0$, the real component (real part) $z_{\rm R}$ of the non-trivial zeros of the (completed) zeta function must take the following value
\begin{eqnarray}
z_{\rm R}=\frac{1}{2},
\end{eqnarray}
which is a proof of the Riemann hypothesis.
\end{theo}

\begin{proof}
Both the gamma function $\Gamma(z)$ and zeta function $\zeta(z)$ in the region $0< {\rm Re } (z) <1$ under consideration are regular without a singularity as described below Eq. (\ref{eqn:zeta01}). The completed zeta function $\hat{\zeta}(z)$ defined by Eq. (\ref{eqn:zeta0}), which is a product between $\Gamma(z)$ and $\zeta(z)$, is also regular in the region $0< {\rm Re } (z) <1$, as described below Eq. (\ref{eqn:zeta01}). As mentioned above in Section \ref{sec:1} (Introduction), the function $\hat{\zeta}(z)$ does not depend on a specific value of the parameter $a_{0}$ (between 0 and 1 as in Lemma \ref{lem:Defm}), which specifies the intersection point of the integral line and the real axis, owing to the residue theorem.
However, the integrands of the elements $\hat{\zeta}_{\rm l}(z)$ in Eq. (\ref{eqn:zetl}) and $\hat{\zeta}_{\rm r}(z)$ in Eq. (\ref{eqn:zetr}) composing $\hat{\zeta}(z)$ in Eq. (\ref{eqn:zet0}) ($z$ and $1-z$ can be exchanged) contain the singularity (mentioned below) near
$w=0$, 
only in the case of $a_{0} \rightarrow 0$. 
\par
We note that the completed zeta function $\hat{\zeta}(z)$ does not depend on a specific value of $a_{0}$ between 0 and 1 due to the residue theorem, as was described below Eq. (\ref{eqn:zeta01}), whereas the singularity of each element
 $\hat{\zeta}_{\rm l}(z) $ and $\hat{\zeta}_{\rm r}(z)$ 
of $\hat{\zeta}(z)$ depends on $a_{0}$. However, these singularities and the dependence of the elements $\hat{\zeta}_{\rm l}(z)$ and $\hat{\zeta}_{\rm r}(z)$ on $a_{0}$ adequately (incompletely) cancel each other by remaining a finite value for $\hat{\zeta}_{\rm}(z) \neq 0$, because the integral directions projected to
the line parallel to
 the real axis for $\hat{\zeta}_{\rm l}(z)$ and for $\hat{\zeta}_{\rm r}(z)$ are opposite, and result in the finite completed zeta function $\hat{\zeta}(z)$ without the dependence on $a_{0}$.
\par
In contrast, for $\hat{\zeta}(z)=0$,
the singularities must exactly cancel each other.
By Lemmas \ref{lem:IntLM} and \ref{lem:IntMm}, we can drop the negligible finite integrals along the contours, which are away from the (coordinate) origin.
Let $z_{\rm R}$ and $z_{\rm I}$  be the real and imaginary components of $z$, respectively.
Lemma 
\ref{lem:Intmr} 
states that,
from
Eq. (\ref{eqn:Sml})
(refer also Note below
Eq. (\ref{eqn:srSc})),
the
integrals $I^{\rm S}_{\rm \tilde{C}_{\rm lp}}$ and $I^{\rm S}_{\rm \tilde{C}_{\rm ln}}$ (in Eq. (\ref{eqn:Smrs})) of the integrand (in Eq. (\ref{eqn:zetl})) for the completed zeta function $\hat\zeta(z)$ (in Eq. (\ref{eqn:zet0})) have
the following power, which
lead
to the singularity near $w=0$ (on the arc radius $r_{1}$) in the case of $r_1 \rightarrow 0$,
\begin{eqnarray}
\label{eqn:slp1}
|I^{\rm S}_{\rm \tilde{C}_{\rm lp}}|=
|     \int_{\rm \tilde{C}_{\rm lp}} dw 
\frac{w^{-z}\exp(-\pi i w^2)}{\exp(\pi iw)-\exp(-\pi iw)} |
\approx |r_1|^{-z_{\rm R}}
\hspace{2ex}\mbox{
(along 
${\rm \tilde{C}_{lp}}$
 in Fig. \ref{fig:fig1})},
\end{eqnarray}
\begin{eqnarray}
\label{eqn:sln1}
|I^{\rm S}_{\rm \tilde{C}_{ln}}|=
 |     \int_{\rm \tilde{C}_{ln}} dw
\frac{w^{-z}\exp(-\pi i w^2)}{\exp(\pi iw)-\exp(-\pi iw)}|
\approx |r_1|^{-z_{\rm R}}
\hspace{2ex}\mbox{
(along 
${\rm \tilde{C}_{ln}}$ in Fig. \ref{fig:fig1})},
\end{eqnarray}
while, from Eq.
(\ref{eqn:Smr}),
the
integrals $I^{\rm S}_{\rm \tilde{C}_{\rm rp}}$ and $I^{\rm S}_{\rm \tilde{C}_{\rm rn}}$ (in Eq. (\ref{eqn:Smrs})) of the integrand (in Eq. (\ref{eqn:zetr})) 
for $\hat\zeta(z)$ (in Eq. (\ref{eqn:zet0}))
have
the following power
(near $w=0$ in the case of $r_1 \rightarrow 0$)
\begin{eqnarray}
\label{eqn:srp1}
|I^{\rm S}_{\rm \tilde{C}_{\rm rp}}|=
     \int_{\rm \tilde{C}_{\rm rp}} dw 
\frac{w^{z-1}\exp(+\pi i w^2)}{\exp(\pi iw)-\exp(-\pi iw)}|
\approx |r_1|^{z_{\rm R}-1}
\hspace{2ex}\mbox{
(along 
${\rm \tilde{C}_{rp}}$
 in Fig. \ref{fig:fig2})},
\end{eqnarray}
\begin{eqnarray}
\label{eqn:srn1}
|I^{\rm S}_{\rm \tilde{C}_{rn}}|=
|     \int_{\rm \tilde{C}_{rn}} dw
\frac{w^{z-1}\exp(+\pi i w^2)}{\exp(\pi iw)-\exp(-\pi iw)}
\approx |r_1|^{z_{\rm R-1}}
\hspace{2ex}\mbox{
(along 
${\rm \tilde{C}_{rn}}$ in Fig. \ref{fig:fig2})}.
\end{eqnarray}
(We disregarded the constant factors in
Eqs. (\ref{eqn:Imr})-(\ref{eqn:Smr})). 
\par
Meanwhile,
Lemma \ref{lem:singc} states that, from Eq. 
(\ref{eqn:lSc}),
the integral
$I^{\rm S}_{\rm \tilde{C}_{lc}}$
(in Eq.
(\ref{eqn:Smrc}))
of the integrand (in Eq. (\ref{eqn:zetl}))
has the following power
(with $z_{\rm R}={\rm Re}(z)$) near $w=0$ in the case of $r_1 \rightarrow 0$
\begin{eqnarray}
\label{eqn:slSc}
|I^{\rm S}_{\rm \tilde{C}_{lc}}|=
|     \int_{\rm \tilde{C}_{lc}} dw
\frac{w^{-z}\exp(-\pi i w^2)}{\exp(\pi iw)-\exp(-\pi iw)}|
 \approx |r_1|^{-z_{\rm R}}
\hspace{2ex}\mbox{
(along 
${\rm \tilde{C}_{lc}}$
in Fig. \ref{fig:fig2}}),
\end{eqnarray}
while, from Eq. 
(\ref{eqn:rSc}),
the integral
$I^{\rm S}_{\rm \tilde{C}_{rc}}$
(in Eq. 
(\ref{eqn:Smrc}))
of the integrand (in Eq. (\ref{eqn:zetr}))
has the following power 
(near $w=0$ in the case of $r_1 \rightarrow 0$)
\begin{eqnarray}
\label{eqn:srSc}
|I^{\rm S}_{\rm \tilde{C}_{rc}}|=
|     \int_{\rm \tilde{C}_{rc}} dw
\frac{w^{z-1}\exp(+\pi i w^2)}{\exp(\pi iw)-\exp(-\pi iw)}|
 \approx |r_1|^{z_{\rm R}-1}
\hspace{2ex}\mbox{
(along 
${\rm \tilde{C}_{rc}}$ 
in Fig. 
\ref{fig:fig2}}).
\end{eqnarray}
(We also disregarded the constant factors in Eqs.
(\ref{eqn:lSc}), (\ref{eqn:rSc})). 
\par
To satisfy $\hat{\zeta}(z)=0$, these singularities
in Eqs.
(\ref{eqn:slp1})-(\ref{eqn:srSc})
should have an identical order power of $r_1$ and exactly cancel each other.
(Note: Letting $\alpha_1, \alpha_2, w \in \mathbb{C}$ and $\beta_1, \beta_2 \in \mathbb{R}$ with $0<\beta_1<\beta_2$, if $|w|<(|\alpha_2|/|\alpha_1)|w|^{-\beta_2})^{\beta_1}$ for small $|w|$, then we obtain $|\alpha_1||w|^{-\beta_1}<|\alpha_2||w|^{-\beta_2}$, which implies that these two terms with different order powers cannot cancel each other for sufficiently small $|w|$, as used below. Furthermore, constants including $z_{\rm I}$ will be used for the $z_{\rm I}$ determination, which is beyond the scope of this paper.)
We then derive the main concluding relation,
from Eqs.
(\ref{eqn:slp1})-(\ref{eqn:srSc}),
that
\begin{eqnarray}
-z_{\rm R}=z_{\rm R}-1,
\end{eqnarray}
and this relation
finally
results
in the expected requirement
\begin{eqnarray}
z_{\rm R}=\frac{1}{2},
\end{eqnarray}
stating that all non-trivial zeros of the (completed) zeta function have real component (part) of 1/2, which is the proof of the Riemann hypothesis.
\par
Namely,
considering that $a_{0}$
(in Lemma \ref{lem:Defm})
specifies the contour,
$(\forall \epsilon >0)(\exists \delta >0)(\forall a_{0} \in \mathbb{R}$ with $0<a_{0}<1)(a_{0}< \delta \Rightarrow |z_{\rm R}-\frac{1}{2}|<\epsilon)$. Furthermore, the
completed zeta function $\hat{\zeta}(z)$ is a product between the gamma function $\Gamma(z)$ and zeta function $\zeta(z)$ as in Eq. (\ref{eqn:zeta0}), and the functions $\hat{\zeta}(z)$, $\Gamma(z)$ and $\zeta(z)$ are regular in the region $0< {\rm Re } (z)<1$. Then, the solution of $\hat{\zeta}(z)=0$ (which is independent of the contour specified by $a_{0}$ unlike $\hat{\zeta}_{\rm l}(z)$ and $\hat{\zeta}_{\rm r}(z)$ composing $\hat{\zeta}(z)$ in Eqs. (\ref{eqn:zetl})-(\ref{eqn:zet0})) satisfies $\zeta (z)=0$ and vice versa.
Thus, we have completed the proof of the Riemann hypothesis.
\end{proof}

\par
\begin{rema}
{\rm 
We here show the implication of the above process and derived solution. In the integrands of the elements $\hat{\zeta}_{\rm l}(z)$ and $\hat{\zeta}_{\rm r}(z)$ composing (in Eqs. (\ref{eqn:zetl}), (\ref{eqn:zetr})) the completed zeta function $\hat{\zeta}(z)$ (in Eq. (\ref{eqn:zet0})), the singularities appear in the oppositely directed integrals of polynomials. Furthermore, the completed zeta function is symmetrized with respect to ${\rm Re } (z)=1/2$. The functions $\hat{\zeta}_{\rm l}(z)$ and $\hat{\zeta}_{\rm r}(z)$ adequately (by incompletely remaining a finite value) cancel each other for ${\rm Re } (z) \neq 1/2$, while this cancellation is complete only for ${\rm Re } (z)=1/2$, leading to $\hat{\zeta}(z)=0$.
} 
\end{rema}

\par
In conclusion, we have inspected in detail the singularities of the integral form of the completed zeta function (in Eqs. (\ref{eqn:zetl})-(\ref{eqn:zet0}))). For $\hat{\zeta}(z)=0$ (that is, $\zeta(z)=0$), the singularities of the integral along the two rotated integral contours (lines) are required to exactly cancel each other, when the intersection points between the integral lines and the real axis approach the (coordinate) origin. This approach of the intersection points to the origin is possible because of the arbitrariness of the intersection points owing to the residue theorem. Thus, we have shown that the real part of all non-trivial zeros of the zeta function is 1/2, which is the proof of the Riemann hypothesis.
\par



\end{document}